\documentclass[11pt,a4paper]{amsart}

\usepackage{graphicx}
\usepackage{nicefrac}
\usepackage{amsmath}
\usepackage{amsfonts}
\usepackage{amssymb}
\usepackage{amsthm}
\usepackage{pst-all}
\usepackage{url}
\usepackage{mathrsfs}
\usepackage{yhmath}
\usepackage{subfigure}
\input xy
\xyoption{all}

\newtheorem{theorem}{Theorem}
\newtheorem*{theorem*}{Theorem}
\newtheorem{corollary}[theorem]{Corollary}
\newtheorem{definition}[theorem]{Definition}
\newtheorem{lemma}[theorem]{Lemma}
\newtheorem*{lemma*}{Lemma}

\newtheorem{proposition}[theorem]{Proposition}
\newtheorem*{proposition*}{Proposition}

\theoremstyle{definition}
\newtheorem{remark}[theorem]{Remark}
\newtheorem{example}[theorem]{Example}

\thanks{The author is funded by the Spanish Ministerio de Ciencia e Innovaci\'on through a Ram\'on y Cajal grant (RYC2018-025843-I)}

\begin{document}

\author{J. J. S\'anchez-Gabites}
\address{Facultad de Ciencias Matem{\'a}ticas. Universidad Complutense de Madrid. 28040 Madrid (Espa{\~{n}}a)}
\email{jajsanch@ucm.es}

\subjclass[2020]{37B35, 37E99, 57K10}

\title[]{Using an invariant knot of a flow to find additional invariant structure}

\begin{abstract} Consider a continuous flow in $\mathbb{R}^3$ or any orientable $3$--manifold. Let $N \subseteq \mathbb{R}^3$ be a compact $3$--manifold such that trajectories of the flow cross $\partial N$ inwards or outwards transversally, or bounce off it from the outside. Suppose we know there exists an invariant knot or link $K$ in the interior of $N$. We prove a generalization of the following: if $K$ is contractible and nontrivial (in the sense of knot theory) in $N$, then every neighbourhood $U$ of $K$ contains a point $p \in N - K$ such that the whole trajectory of $p$ is contained in $N$. In other words, there is additional invariant structure in $N$ besides $K$ and it can be found arbitrarily close to $K$.

To prove this result we develop a ``coloured'' handle theory which may be of independent interest to study flows in $3$-manifolds.
\end{abstract}

\begin{abstract} Consider a continuous flow in $\mathbb{R}^3$ or any orientable $3$--manifold. Let $(Q_1,Q_0)$ be an index pair in the sense of Conley and consider the region $N := \overline{Q_1-Q_0}$. (An example of this is a compact $3$--manifold $N$ such that trajectories of the flow cross $\partial N$ inwards or outwards transversally, or bounce off it from the outside). Suppose we know there is an invariant knot or link $K$ in the interior of $N$. We prove the following: if $K$ is contractible and nontrivial (in the sense of knot theory) in $N$, then every neighbourhood $U$ of $K$ contains a point $p \in N - K$ such that the whole trajectory of $p$ is contained in $N$. In other words, the presence of $K$ forces the existence of additional invariant structure in $N$ (besides $K$), and the latter can actually be found arbitrarily close to $K$.

To prove this result we develop a ``coloured'' handle theory which may be of independent interest to study flows in $3$-manifolds.
\end{abstract}

\maketitle

\section{Introduction}

In the theory of dynamical systems there exist several tools to establish the existence of invariant structure in a given region $N$ of phase space; for instance the fixed point index or the Conley index. If one already knows about some nonempty invariant set $K \subseteq N$ then these tools can also be used to infer the existence of additional invariant structure in $N$; i.e. to show that there exist points $p \in N - K$ whose full orbit is entirely contained in $N$. This is the overall flavour of the main theorem of the paper: for a flow in an orientable $3$--manifold we will show that the presence of a contractible invariant knot or link $K$ in $N$ forces, under appropriate assumptions, the existence of additional invariant structure.
\smallskip

A) We first present the main theorem in a simplified way. Let $\varphi$ be a continuous flow in the phase space $\mathbb{R}^3$. We abbreviate $\varphi(p,t)$ by $p \cdot t$, where $t$ is the time variable. Our region of interest will be a compact $3$--manifold $N \subseteq \mathbb{R}^3$ and, as usual, one has to impose some condition on the behaviour of the flow on its boundary $\partial N$. We borrow one from Conley and Easton (\cite{conleyeaston1}): $\partial N$ must be the union of two closed sets $N^o$ and $N^i$ (for ``out'' and ``in'') such that points in $N^o$ flow towards the outside of $N$ and points in $N^i$ flow from it. Formally, for each $p \in N^o$ there exists $\epsilon > 0$ such that $p \cdot (0,\epsilon) \cap N = \emptyset$ and for each $p \in N^i$ there exists $\epsilon > 0$ such that $p \cdot (-\epsilon,0) \cap N = \emptyset$. At points $p \in N^i \cap N^o$ both conditions hold so trajectories bounce off $N$. For example, in a smooth setting where the flow is generated by a vectorfield $X(p)$ and $N$ is a smooth manifold, this amounts to requiring that at points $p$ where $X(p)$ is tangent to $\partial N$ the orbit through $p$ is externally tangent to $N$. (Let $\nu(p)$ be an outward normal vector at each point of $\partial N$ and set $N^o$ and $N^i$ to be the subsets of $\partial N$ defined by $X(p) \cdot \nu(p) \geq 0$ and $X(p) \cdot \nu(p) \leq 0$ respectively).

The following definitions are standard. A knot $K$ is a (tame) simple closed curve. It is the trivial knot if it bounds an embedded disk $D$ (equivalently, if it can be carried by a homeomorphism of $\mathbb{R}^3$ onto some model of an unknotted curve; say the unit circle in the $xy$ plane). If $K \subseteq N$, we say that the knot $K$ is trivial in $N$ if it bounds an embedded disk $D \subseteq N$. A link $K$ is a union of finitely many disjoint knots $K_i$; if $K \subseteq N$ we say that the link $K$ is trivial in $N$ if there exist disjoint embedded disks $D_i \subseteq N$ such that each $K_i$ bounds the disk $D_i$.

A simplified version of the main theorem reads as follows:

\begin{theorem*}Let $N$ be a region as described. Assume that $N$ contains an invariant link $K$ which is contractible and nontrivial in $N$. Then every neighbourhood $U$ of $K$ contains a point $p \in U - K$ such that the full trajectory of $p$ is contained in $N$.
\end{theorem*}

Hence not only there exist full orbits in $N$ disjoint from $K$ (namely the orbit through $p$), but in fact these can be found passing arbitrarily close to $K$. An equivalent way of expressing the conclusion of the theorem is that at least one of the following holds:
\begin{itemize}
    \item There exists at least one trajectory disjoint from $K$, entirely contained in $N$, and whose $\alpha$- or $\omega$-limit is contained in $K$.
    \item There exist infinitely many trajectories disjoint from $K$ and entirely contained in any given neighbourhood $U$ of $K$ (hence in $N$).
\end{itemize}

Figure \ref{fig:basicexample} shows four examples of links $K$ in a compact $3$--manifold $N$ (shaded). In all of them $K$ is nontrivial inside $N$: $K$ is a trefoil knot in panels (a) and (d) and a Hopf link in panel (b); since these are nontrivial in $\mathbb{R}^3$ they are also nontrivial in the smaller set $N$. In panel (c) the knot $K$ is in fact trivial in $\mathbb{R}^3$ but not in $N$ because together with a meridian of the torus it becomes the Whitehead link, which is nontrivial in $\mathbb{R}^3$. In panels (a)-(c) $K$ is contractible in $N$; in panel (d) it is not. Hence in panels (a)-(c) the geometric conditions of the theorem are satisfied, so if $K$ is invariant for some flow satisfying the boundary conditions on $\partial N$, there must exist additional invariant structure inside $N$ and passing arbitrarily close to $K$. In panel (d) it is very easy to construct a flow having $K$ as a stable (local) attractor comprised of fixed points, with $N$ a positively invariant neighbourhood whose boundary is traversed by the flow inwards everywhere. Then there is no other invariant structure in $N$ besides $K$, showing that the theorem may fail if $K$ is not contractible in $N$.

\begin{figure}[h!]
\null\hfill
\subfigure[]{
\begin{pspicture}(0,0)(3.2,3.2)
	\rput[bl](0,0){\scalebox{0.6}{\includegraphics{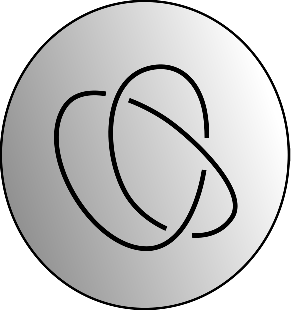}}}
    \rput[bl](2.5,0.1){$N$}
\end{pspicture}}
\hfill
\subfigure[]{
\begin{pspicture}(0,0)(3.2,3.2)
	\rput[bl](0,0){\scalebox{0.6}{\includegraphics{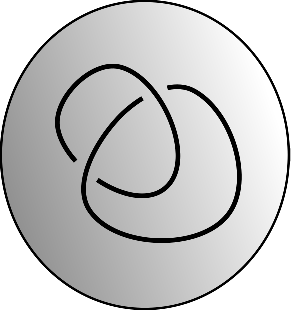}}}
    \rput[bl](2.5,0.1){$N$}
\end{pspicture}}
\hfill
\subfigure[]{
\begin{pspicture}(0,0)(6.2,2.5)
	\rput[bl](0,0){\scalebox{0.6}{\includegraphics{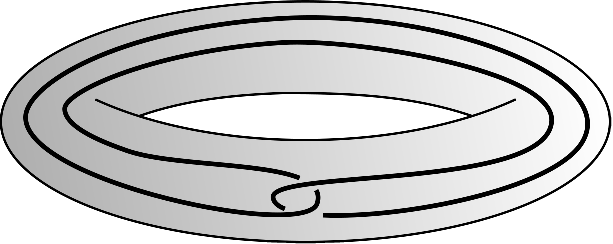}}}	
 \rput[bl](5.5,0.1){$N$}
\end{pspicture}}
\hfill
\subfigure[]{
\begin{pspicture}(0,0)(2.2,2.2)
	\rput[bl](0,0){\scalebox{0.6}{\includegraphics{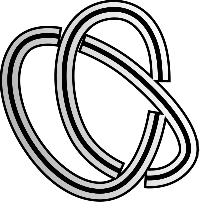}}}	
 \rput[bl](1.8,1.8){$N$}
\end{pspicture}}
\hfill\null
\caption{ \label{fig:basicexample}}
\end{figure}

We focus on Figure \ref{fig:basicexample}.(a) and consider a very simple flow that runs vertically downwards and has the points in $K$, and only them, as equilibria. Then $N^i$ and $N^o$ are the closed upper and lower hemispheres of $\partial N$, so the theorem applies and we expect additional invariant structure in $N$. And indeed, if we look down at the ball $N$ from its north pole we see a planar projection of $K$ with three crossings; each of these corresponds to two points in $K$ stacked vertically and then the straight line segment $\gamma_i$ they determine (minus its endpoints) is a full trajectory of the flow which is contained in $N$ and disjoint from $K$. Notice that in this example $K$ is an isolated invariant set and has a well defined Conley index, so one may attempt to locate the additional invariant structure in $N$ by comparing the Conley index of $K$ with the homotopy type of $(N/N^o,[N^o])$, which is trivial. However, in this case the index of $K$ is also trivial and so the trajectories $\gamma_i$ detected by the theorem are invisible to the Conley index. In fact, since an arbitrarily small perturbation of the flow will remove all invariant structure in $N$ (just make all trajectories run vertically downwards undisturbed), this will happen for any index that is invariant under small perturbations. Finally, we can modify this example to illustrate the role of the ``no interior tangencies'' condition placed on $N$. Remove from $N$ three small open balls $B_i$, each centered at a point lying midway of $\gamma_i$. The resulting perforated manifold has interior tangencies; any point in the equator of each $2$--sphere $\partial B_i$ consists entirely of such points. Obviously $K$ is still nullhomotopic (and knotted) in this perforated $N$, but now it is the only invariant structure in it.

Notice that the theorem makes no assumption on the dynamics on $K$ or the interior of $N$, and these can be very degenerate (far from hyperbolic, for instance). Also, the entry and exit pattern of the flow through $\partial N$ plays no role (unlike in Conley index theory).

\smallskip

B) We motivate the general form of the main theorem as follows. Finding a region $N$ with the neat structure described above might be difficult, so we now suppose that $N$ is a compact but otherwise arbitrary subset of $\mathbb{R}^3$ whose interior we know contains an invariant link $K$. As before, we are looking for additional invariant structure in $N$. One of the fundamental theorems of Conley index theory states that one of these holds: (i) either the trajectory of some $p$ in the frontier of $N$ is entirely contained in $N$ or (ii) there exists a (nonunique) compact pair $(Q,Q_0)$ in $N$ with certain dynamical properties (a so-called index pair) such that the interior of $\overline{Q-Q_0}$ contains all the invariant structure in $N$. Index pairs can be explicitly computed in favorable circumstances; for instance, for smooth flows. Now, if (i) holds we are finished: the trajectory of $p$ is disjoint from $K$ and entirely contained in $N$. The interesting situation is case (ii). Since $\overline{Q-Q_0}$ contains $K$ and any possible additional invariant structure in $N$, we may as well replace the initial region $N$ with this new, smaller region $\overline{Q - Q_0}$. Although this will not generally be a manifold and the behaviour of the flow on its frontier might be quite complicated, the theorem stated above remains essentially true. This is the content of our main theorem:

\begin{theorem} \label{thm:detect1} Let $(Q,Q_0)$ be an index pair in $\mathbb{R}^3$. Assume that $\overline{Q-Q_0}$ contains an invariant link $K$ which is contractible and nontrivial in its interior. Then every neighbourhood $U$ of $K$ contains a point $p \in U - K$ such that the full trajectory of $p$ is contained in $Q-Q_0$.
\end{theorem}

(The simplified version above follows from this by setting $(Q,Q_0) = (N,N^o)$).

\smallskip

C) To prove Theorem \ref{thm:detect1} we develop some tools which are of independent interest. The argument is by contradiction and very roughly goes as follows. If the conclusion does not hold, then the maximal (biggest) invariant subset of $N$ is the disjoint union of $K$ and another compact invariant set $K'$ (in principle, possibly empty). We will show that $K$ and $K'$ have disjoint neighbourhoods $T$ and $B$ which are compact $3$--manifolds and have some special properties: (i) at every point in their boundaries the flow enters or exits the neighbourhood transversally or is externally tangent to it, (ii) $T$ is a union of disjoint solid tori along the components of $K$, (iii) $N$ can be obtained from $T \cup B$ by attaching handles onto it. It is the latter property that is key to our arguments. This manner of enlarging a manifold is well known from differential and piecewise linear topology; however, in our case the handle attachments will be heavily constrained by the dynamics. Imagine each point in the boundary of $T \cup B$ coloured gray or white depending on whether the trajectory through the point exits or enters $T \cup B$. The handles also have standard colourings and have to be pasted onto $T \cup B$ in such a manner that colours match. The proof of the theorem will then be translated to a geometric argument showing that when the link $K$ is nontrivial (and hence so is $T$) it is not possible to paste coloured handles onto $T \cup B$ to obtain $N$.

The paper is organized as follows. Section \ref{sec:prelim} contains some background definitions. Section \ref{sec:iblocks} introduces isolating blocks and regular isolating blocks. These are neighbourhoods of invariant sets that have particularly nice properties; for instance the neighbourhood $T$ mentioned above is a regular isolating block for $K$. Section \ref{sec:coloured} introduces coloured manifolds and describes the pasting of coloured handles. Section \ref{sec:asas} proves the coloured handle theorem; the key step (iii) in the argument outlined above. Section \ref{sec:concentric} shows how to recognize the topological type of regular isolating blocks. All this machinery is developed in general and then applied to prove Theorem \ref{thm:detect1} in Section \ref{sec:proof}. A very brief Section \ref{sec:concluding} contains some concluding remarks regarding the machinery of coloured handles and explains how the dynamical problem of finding additional invariant structure in a manifold $N$ as considered in A) is essentially equivalent to a purely topological problem of deciding whether a coloured manifold can be obtained from another one by pasting coloured handles onto it.

\section{Background definitions} \label{sec:prelim}

\subsection{} Throughout the paper we will take $\mathbb{R}^3$ as our phase space for simplicity but all proofs work exactly the same in any orientable boundariless $3$--manifold. Recall that a set $A \subseteq \mathbb{R}^3$ is called tame if there exists a homeomorphism of $\mathbb{R}^3$ that sends $A$ onto a polyhedron; i.e. (the polytope of) a finite simplicial complex in $\mathbb{R}^3$.

As a notational convention, the topological frontier and interior of a subset $A$ will be denoted by ${\rm fr}\ A$ and ${\rm int}\ A$ (lower case) respectively. A manifold will always mean a topological manifold possibly with boundary. The boundary of a manifold $A$ will be denoted by $\partial A$ and its interior by ${\rm Int}\ A$ (upper case).

We shall use \v{C}ech cohomology $\check{H}$ because it is better suited to compacta with a possibly complicated structure. The following continuity property is particularly useful in dynamics: if $K$ is the intersection of a nested sequence of compacta $Y_1 \supseteq Y_2 \supseteq \ldots$  then $\check{H}^*(K)$ is the direct limit of the direct sequence \[\check{H}^*(Y_1) \longrightarrow \check{H}^*(Y_2) \longrightarrow \ldots\] where all arrows are induced by inclusion (see \cite[Theorem 6, p. 318]{spanier1}). The following is also useful: $\check{H}^*(A,B) = \check{H}^*(A/B,[B])$ for any compact pair (a direct consequence of \cite[Lemma 11, p. 321]{spanier1}).

Only a very elementary knowledge of knot theory will be needed, with Rolfsen \cite{rolfsen1} being an appropriate reference.

\subsection{Isolating neighbourhoods} \label{subsec:inbd} The definitions (unfortunately not the notations) to come are all standard in Conley index theory and can be found for example in his monograph \cite{conley1}.

Let $N$ be a compact set. Its maximal (biggest) invariant subset is ${\rm Inv}(N) = \{p \in N : p \cdot \mathbb{R} \subseteq N\}$; we shall usually denote it by $K$. This is always closed in $N$, hence compact if nonempty. $N$ is called an isolating neighbourhood when the orbit of every $p \in {\rm fr}\ N$ leaves $N$ either in forward or backward time. This ensures that $K$ is contained in the interior of $N$. Changing perspectives, a compact invariant set $K$ is called isolated if it is the maximal invariant subset of some isolating neighbourhood $N$.

For an isolating neighbourhood $N$ we set \[N^{+} := \{p \in N : p \cdot [0,+\infty) \subseteq N\} \quad  \text{ and } \quad N^{-} := \{p \in N : p \cdot [0,+\infty) \subseteq N\};\] notice that for $p \in N^+$ we must have $\emptyset \neq \omega(p) \subseteq K$ and similarly for $p \in N^-$. It is clear that both are compact (if nonempty) and $N^{+} \cap N^{-} = K$. We also set $n^+ := N^+ \cap {\rm fr}\ N$ and $n^{-} := N^{-} \cap {\rm fr}\ N$.

Given a point $p \in N$ set \[t^o(p) := \sup\ \{t \geq 0: p \cdot [0,t] \subseteq N\}\] and \[t^i(p) := \inf\ \{t \leq 0 : p \cdot [t,0] \subseteq N\}.\] Thus $t^i(p)$ and $t^o(p)$ are the (signed) times that $p$ needs to exit $N$ for the first time in the past or in the future. Both are allowed to take the values $-\infty$ and $+\infty$, which happens if and only if $p \in N^{-}$ or $p \in N^{+}$, respectively. The maximal trajectory segment of $p$ in $N$ is defined as $p \cdot I$ where $I$ is the closed interval determined by $t^i(p)$ and $t^o(p)$ (possibly infinite); its endpoints are $p \cdot t^i(p)$ and $p \cdot t^o(p)$ when defined. The immediate entry and exit sets of $N$ are defined as \[N^i := \{p \in N : t^i(p) = 0\} \quad \text{ and } \quad N^o := \{p \in N : t^o(p) = 0\}\] so $p$ is an immediate exit point if, and only if, $p \cdot [0,\epsilon) \not\subseteq N$ for every $\epsilon > 0$, and similarly for immediate entry points. Clearly $N^i$ and $N^o$ are always subsets of ${\rm fr}\ N$. Finally, the projections onto $N^i$ and $N^o$ \[\pi^i : N - N^{-} \longrightarrow N^i \text{ and } \pi^o : N - N^{+} \longrightarrow N^o\] are defined by $\pi^i(p) := p \cdot t^i(p)$ and $\pi^o(p) := p \cdot t^o(p)$ respectively. Geometrically $\pi^i(p)$ is the point through which the trajectory of $p$ enters $N$, while $\pi^o(p)$ is the point through which the trajectory of $p$ leaves $N$.

\begin{remark} To aid visualization, let us say that a point $p \in {\rm fr}\ N$ is a ``clean'' exit if there exists $\epsilon>0$ such that $p \cdot (0,\epsilon) \cap N = \emptyset$. We claim that the set of clean exit points is dense in $N^o$. It is evidently contained in $N^o$. To show density pick $p \in N^o$ and consider $\{t \in (0,1) : p \cdot t \not\in N\} \subseteq \mathbb{R}$. This is an open set each of whose connected components is an open interval; moreover, $0$ belongs to its closure since $p \in N^o$. If there is a component of the form $(0,\epsilon)$ then $p$ is a clean exit point. If not, there exists a sequence of components $(t_n,t_n+\epsilon_n)$ with $t_n \rightarrow 0$. Each $p \cdot t_n$ is a clean exit point and they converge to $p$.
\end{remark}

\subsection{} In general the maps $t^i$, $t^o$, $\pi^i$ and $\pi^o$ need not be continuous. The following conditions are equivalent:
\begin{itemize}
	\item[(i)] $N^o$ is closed in $N$.
	\item[(ii)] The mapping $t^o : N \longrightarrow [0,+\infty]$ is continuous.
	\item[(iii)] $\pi^o$ is continuous.
\end{itemize}

Of course, analogous equivalences hold for $N^i$, $t^i$ and $\pi^i$. (ii) $\Rightarrow$ (iii) is trivial, (iii) $\Rightarrow$ (i) is a consequence of the fact that $N^o$ is precisely the fixed point set of $\pi^o$, and (i) $\Rightarrow$ (ii) dates back to Wa\.zewski \cite{wazewski1} (or see the proof of \cite[Theorem 2.3, pp. 24 and 25]{conley1}). We shall use these equivalences tacitly in the sequel. These conditions are related to how (maximal) trajectory segments in $N$ look like:

\begin{proposition} \label{prop:structure} Let $N$ be an isolating neighbourhood with a closed $N^o$. Then for every $p \in N$ there exists $t \in [t^i(p),t^o(p)]$ such that $p \cdot (t^i(p),t) \subseteq {\rm fr}\ N$ and $p \cdot (t,t^o(p)) \subseteq {\rm int}\ N$.
\end{proposition}

Figure \ref{fig:structure} shows some examples. A trajectory segment in $N$ can travel initially along ${\rm fr}\ N$, then enter the interior of $N$ and remain there until it hits ${\rm fr}\ N$ again, at which point it actually hits $N^o$ and therefore exits $N$ (perhaps not cleanly, suggested by the dashed trajectory). One can have $t = t^i(p)$ or $t = t^o(p)$, so the segment might be entirely contained in ${\rm fr}\ N$ or entirely contained in ${\rm int}\ N$ (save for its endpoints). For example if $t^i(p) = -\infty$ (i.e. $p \in N^-$) then its distant past is close to $K$ and therefore contained in ${\rm int}\ N$, so we must have $t = t^i(p) = -\infty$ and $p \cdot (-\infty,t^o(p)) \subseteq {\rm int}\ N$.

\begin{figure}[h!]
\null\hfill
\subfigure[$-\infty < t^i < t^o < +\infty$]{
\begin{pspicture}(0,0)(4.2,2.2)
	\rput[bl](0.5,0){\scalebox{0.6}{\includegraphics{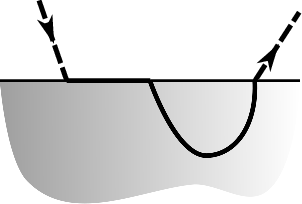}}}
\end{pspicture}}
\hfill
\subfigure[$-\infty < t^i < t^o = +\infty$]{
\begin{pspicture}(0,0)(4.2,2.2)
	\rput[bl](0.5,0){\scalebox{0.6}{\includegraphics{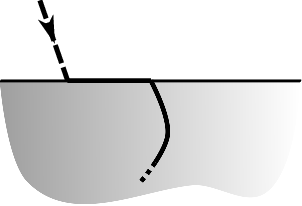}}}	
\end{pspicture}}
\hfill
\subfigure[$-\infty = t^i < t^o <+\infty$]{
\begin{pspicture}(0,0)(4.2,2.2)
	\rput[bl](0.5,0){\scalebox{0.6}{\includegraphics{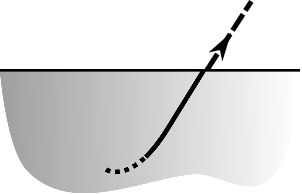}}}	
\end{pspicture}}
\hfill\null
\caption{ \label{fig:structure}}
\end{figure}

 \begin{proof}[Proof of Proposition \ref{prop:structure}] We first establish the following:

{\it Claim.} Let $p \in {\rm fr}\ N$ satisfy $p \cdot (-\epsilon,0] \cap {\rm int}\ N \neq \emptyset$ for every $\epsilon > 0$. Then $p \in N^o$.

{\it Proof.} By assumption there exists a sequence $-\nicefrac{1}{n} \leq t_n < 0$ such that $p \cdot t_n \in {\rm int}\ N$ for each $n$. Choose neighbourhoods $U_n$ of $p$ in phase space such that $U_n \cdot t_n \subseteq {\rm int}\ N$. Since $p \in {\rm fr}\ N$ there exists a sequence $p_n \rightarrow p$ such that $p_n \in U_n - N$ for each $n$. Then $p_n \cdot t_n \in {\rm int}\ N \subseteq N$ but $p_n \not\in N$ so $t^o(p_n \cdot t_n) \leq |t_n| \leq \nicefrac{1}{n}$. Then $t^o(p) = \lim t^o(p_n \cdot t_n) = 0$ by the continuity of $t^o$ at $p$. $_{\blacksquare}$
\smallskip

To prove the proposition let $p \in N$ and $U = \{s \in (t^i(p),t^o(p)) : p \cdot s \in {\rm int}\ N\}$. If $U = \emptyset$ then the proposition holds setting $t = t^o(p)$, so assume $U \neq \emptyset$. Clearly $U$ is open in $(t^i(p),t^o(p))$ and each of is connected components is an open interval whose endpoints do not belong to $U$. Let $J$ be one of these. We claim that its right endpoint $s$ must be $t^o(p)$. If not, $s$ is finite and there exists an increasing sequence $s_n$ in $J$ converging to $s$. By construction $p \cdot s_n \in {\rm int}\ N$ and $p \cdot s \in {\rm fr}\ N$ so by the claim above we have $p \cdot s \in N^o$. This implies $t^o(p) = s < t^o(p)$, a contradiction. Thus $J$ must be an interval of the form $(t,t^o(p))$; since this is true of all components of $U$ there can only be one of them. This concludes the proof.
\end{proof}

\begin{remark} \label{rem:structure} If both $N^i$ and $N^o$ are closed then the same argument for the reverse flow shows that only two possibilities can happen: either $p \cdot (t^i(p),t^o(p)) \subseteq {\rm fr}\ N$ or $p \cdot (t^i(p),t^o(p)) \subseteq {\rm int}\ N$.
\end{remark}

\subsection{} There exist several relations between the cohomology of some of the sets defined above. These are all based on the idea of using the flow to deform, possibly ``in infinite time'', one set onto another. Detailed proofs can be found elsewhere (\cite[\S 4]{churchill1} or  \cite[\S 2]{conleyeaston1}) so we just state the result and provide a very brief idea of the proof. 

\begin{proposition} \label{prop:coh} Let $N$ be an isolating neighbourhood for $K$. Assume that $t^i$ is continuous. Then the following isomorphisms hold:
\begin{itemize}
    \item[(i)] $\check{H}^*(N^+) = \check{H}^*(K)$, induced by the inclusion $K \subseteq N^+$.
    \item[(ii)] $\check{H}^*(N,N^i) = \check{H}^*(N^+,n^+)$.
    \item[(iii)] $\check{H}^*(N,K) = \check{H}^*(N^i,n^+)$.
\end{itemize}
\end{proposition}
\begin{proof} (i) Define the sets $Y_k := N^+ \cdot k$ for $k \geq 0$. These form a nested sequence of compact, positively invariant sets whose intersection is $K$. The map $f : p \longmapsto p \cdot 1$ restricts to a self-map of each $Y_k$ which is homotopic to the ${\rm id}_{Y_k}$ via the flow $(\varphi_t)_{t \in [0,1]}$. Each inclusion $i : Y_{k+1} \subseteq Y_k$ is then a homotopy equivalence, with the map $f$ providing a homotopy inverse. By the continuity property of \v{C}ech cohomology $\check{H}^*(K)$ is the direct limit of the sequence \[\check{H}^*(Y_1) \longrightarrow \check{H}^*(Y_2) \longrightarrow \ldots\] where the arrows are induced by inclusions. These are all isomorphisms, and so in the limit the inclusion $K \subseteq N^+$ also induces an isomorphism $\check{H}^*(N^+) = \check{H}^*(K)$. We shall abuse language and say that the flow produces a deformation ``in infinite time'' of $N^+$ onto $K$.

(ii) Start with $(N,N^i)$ and flow every point backwards until it first hits $N^i$; namely consider $N \times (-\infty,0] \longrightarrow N$ given by $(p,t) \longmapsto p \cdot \max \{t,t^i(p)\}$. This is where continuity of $t^i$ is used. The pair gets deformed in infinite time onto $(N^i \cup N^+, N^i)$ and so $\check{H}^*(N,N^i) = \check{H}^*(N^i \cup N^+,N^i)$. We then have \[\check{H}^*(N,N^i) = \check{H}^*(N^i \cup N^+,N^i) = \check{H}^*(N^+,n^+)\] because $(N^i \cup N^+)/N^i = N^+/n^+$.

(iii) Start with $(N,N^+)$ and again flow every point backwards until first it hits $N^i$. This deforms (again, ``in infinite time'') $(N,N^+)$ onto $(N^i \cup N^+, N^+)$. Thus \[\check{H}^*(N,N^+) = \check{H}^*(N^i \cup N^+, N^+) = \check{H}^*(N^i,n^+)\] Finally, using (i) and the long exact sequence for the triple $(N,N^+,K)$ we have that the inclusion $(N,K) \subseteq (N,N^+)$ induces isomorphism in cohomology and, together with the above, $\check{H}^*(N,K) = \check{H}^*(N^i,n^+)$.
\end{proof}

\section{Isolating blocks} \label{sec:iblocks}

Isolating blocks are a particularly nice class of isolating neighbourhoods. These were introduced in the smooth setting by Conley and Easton \cite{conleyeaston1} and our definition below is an almost direct translation of theirs to our topological setting. We warn the reader that the definition of an ``isolating block'' is not consistent across the literature, and in particular ours is more stringent than that used in \cite{churchill1} and \cite{gierzkiewicz1}.

Let $N \subseteq \mathbb{R}^3$ be a compact $3$--manifold and $p \in \partial N$. We shall say that $p$ is a transverse (i) entry or (ii) exit point if there exists an $\epsilon > 0$ such that either (i) $p \cdot (-\epsilon,0) \cap N = \emptyset$ and $p \cdot (0,\epsilon) \subseteq {\rm int}\ N$, or (ii) $p \cdot (-\epsilon,0) \subseteq {\rm int}\ N$ and $p \cdot (0,\epsilon) \cap N = \emptyset$. Similarly, $p$ is an exterior tangency if $p \cdot (-\epsilon,0) \cap N = \emptyset = p \cdot (0,\epsilon) \cap N$. See points $p_1$, $p_2$, $p_3$ in Figure \ref{fig:boundary} below.

\begin{definition} \label{defn:iblocks} An isolating block $N$ is a tame compact $3$--manifold $N \subseteq \mathbb{R}^3$ whose boundary $\partial N$ is the union of two compact $2$--manifolds $N^i$ and $N^o$ (one may be possibly empty) with common boundary $\partial N^i = \partial N^o = N^i \cap N^o$ and such that:
\begin{itemize}
	\item[(i)] every $p \in {\rm Int}\ N^i$ is a transverse entry point,
	\item[(ii)] every $p \in {\rm Int}\ N^o$ is a transverse exit point,
	\item[(iii)] every $p \in N^i \cap N^o$ is an exterior tangency.
\end{itemize}
\end{definition}

Figure \ref{fig:boundary} shows the setting conveyed by this definition. It depicts a portion of an isolating block, with $N^i$ painted white and $N^o$ painted dark grey, a convention that will hold throughout the paper.

\begin{figure}[h]
\begin{pspicture}(0,-0.2)(7.5,3)
	\rput[bl](0,0){\scalebox{0.8}{\includegraphics{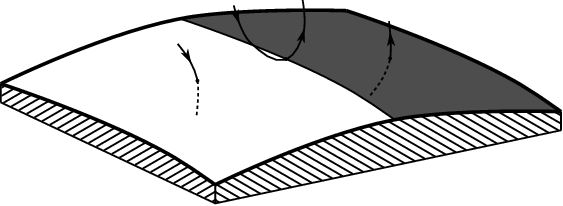}}}
	\rput[bl](2.2,1.5){$p_1$} \rput[bl](5.4,1.8){$p_2$} \rput[l](3.7,1.8){$p_3$}
	\rput(3.0,0.7){$N^i$} \rput(6.7,1.5){$N^o$} \rput(1,0.5){$N$}
\end{pspicture}
\caption{\label{fig:boundary}}
\end{figure}

Let $N$ be an isolating block and $K$ its maximal invariant subset. The condition that there are no interior tangencies to $\partial N$ implies that $K$ is contained in the interior of $N$, and so $N$ is an isolating neighbourhood; it is straightforward to check that its entry and exit sets as defined in Subsection \ref{subsec:inbd} are indeed $N^i$ and $N^o$. Since these are closed by definition, the maps $t^i$, $t^o$, $\pi^i$ and $\pi^o$ are continuous. We will prove later on that any isolated invariant set has a neighbourhood basis of isolating blocks.

In general there need not be any geometric relation between $K$ and $N$. This motivates the following definition: $N$ is a regular isolating block (for $K$) if the inclusion  $i : K \subseteq N$ induces isomorphisms $i^* : \check{H}^*(N;\mathbb{Z}) \longrightarrow \check{H}^*(K;\mathbb{Z})$ in \v{C}ech cohomology. This condition has already been considered in the literature (\cite{easton2}, \cite{gierzkiewicz1}, but recall the caveat about the terminology ``isolating block'') but does not seem to have been given any particular name. We call such an $N$ ``regular'' because of its similarities with a regular neighbourhood in piecewise linear topology (or a tubular neighbourhood in differential topology). This goes beyond a heuristic relation in view of Corollary \ref{cor:regular_tame}. An isolated invariant set $K$ with a finitely generated \v{C}ech cohomology has a neighbourhood basis of regular isolating blocks: this is \cite[Theorem 4.3, p. 321]{gierzkiewicz1}; see also p. \pageref{pg:ribs} for an outline of the proof.

\subsection{The tameness condition} \label{sub:tame} By definition an isolating block $N$ is required to be tame, so there exists a homeomorphism of $\mathbb{R}^3$ that sends $N$ onto a polyhedron. There is a local version of tameness: $N$ is locally tame at a point $p$ if there exist a closed neighbourhood $M$ of $p$ and an embedding $h : M \longrightarrow \mathbb{R}^3$ such that $h(M \cap N)$ is a polyhedron. A very deep theorem of Bing and Moise (see \cite[Theorem 4, p. 254]{moise2} and references therein)
establishes that $N$ is tame if and only if it is locally tame at each point. In our dynamical setting local tameness comes for free near transverse entry and exit points. To explain this we begin with the following purely dynamical result (recall that $\varphi$ is the flow):

\begin{proposition} \label{prop:loc_prod} Let $N$ be an isolating neighbourhood. Assume $O$ is an open subset of ${\rm fr}\ N$ consisting entirely of transverse entry points and let $p \in O$. There exist a neighbourhood $V \subseteq O$ of $p$ and $\delta > 0$ such that $\varphi : V \times (-\delta,\delta) \longrightarrow V \cdot (-\delta,\delta)$ is a homeomorphism and $U := V \cdot (-\delta,\delta)$ is an open neighbourhood of $p$ in $\mathbb{R}^3$ with $U \cap N = V \cdot [0,\delta)$.
\end{proposition}
\begin{proof} Since $p$ is a transverse entry point there exists $\epsilon > 0$ such that $p \cdot [-\epsilon,0) \cap N = \emptyset$ and $p \cdot (0,\epsilon] \subseteq {\rm int}\ N$. Any sufficiently small open neighbourhood $W$ of $p \cdot (-\epsilon)$ in phase space will satisfy (i) $W \cap N = \emptyset$, (ii) $W \cdot (2\epsilon) \subseteq {\rm int}\ N$, (iii) $W \cdot [0,2\epsilon] \cap {\rm fr}\ N \subseteq O$ and so consists only of transverse entry points. Let $U := W \cdot [0,2\epsilon]$ and let $\sigma$ be a maximal trajectory segment in $U$. By definition $\sigma$ must intersect a segment of the form $q \cdot [0,2\epsilon]$ with $q \in W$ and by maximality it must contain it. Then (i) and (ii) above ensure that $\sigma$ intersects ${\rm fr}\ N$ at least once; (iii) ensures that it intersects ${\rm fr}\ N$ only at transverse entry points and therefore it must intersect it just once. Consequently there is a well defined map $\tau : U \longrightarrow [-2\epsilon,2\epsilon]$ such that $q \cdot \tau(q)$ is the unique point of intersection of the trajectory segment containing $q$ with ${\rm fr}\  N$. It is easy to check that $\tau$ is continuous. Notice that $U$ contains $p$ and is open in phase space (it is the union of the open sets $W \cdot t$ for $t \in [0,2 \epsilon]$). Thus there exist an open neighbourhood $V$ of $p$ in ${\rm fr}\  N$ and $\delta > 0$ such that $V \cdot (-\delta,\delta) \subseteq U$. Notice that $q \in V \cdot (-\delta,\delta)$ if and only if $q \cdot \tau(q) \in V$ and $\tau(q) \in (-\delta,\delta)$. Since $\tau$ is continuous and $V$ is open, it follows that $V \cdot (-\delta,\delta)$ is open in $U$ hence also in phase space.
\end{proof}

Suppose that a certain $N$ is known to satisfy all conditions in the definition of an isolating block except perhaps for the tameness condition. By the theorem of Moise it suffices to check that $N$ is locally tame at each $p \in \partial N$ (it is obviously locally tame at each interior point). Suppose $p$ is not a tangency point; assume for instance that it is a transverse entry point. We apply Proposition \ref{prop:loc_prod} and observe that since we know that $\partial N$ is a surface, perhaps after reducing $V$ we may take it to be homeomorphic to $\mathbb{R}^2$. Then $U$ is an open neighbourhood of $p$ and \[(U,U \cap N) \cong V \times ((-\delta,\delta),[0,\delta)) \cong \mathbb{R}^2 \times (\mathbb{R},[0,+\infty)) = (\mathbb{R}^3, \mathbb{R}^2 \times [0,+\infty))\] where the first homeomorphism is provided by the flow and the second uses $V \cong \mathbb{R}^2$. Since $\mathbb{R}^2 \times [0,+\infty)$ is locally tame in $\mathbb{R}^3$, it follows that $N$ is locally tame at $p$. Obviously at a transverse exit point the same argument holds. The conclusion is that to show that $N$ is tame it suffices to check that it is locally tame at tangency points.

\subsection{Cylindrical coordinates} A height function $u$ for $N$ is a continuous mapping defined for every $p \in N$ except for those points in the boundary where the flow is externally tangent to $N$, and with the following properties:

\begin{itemize}
	\item[(i)] $u(p) \in [-1,1]$,
    \item[(ii)] $u(p) = 1 \Leftrightarrow p \in N^i$ and $u(p) = -1 \Leftrightarrow 
 p \in N^o$,
	\item[(iii)] $u|_K \equiv 0$, and
	\item[(iv)] $u$ is strictly decreasing along the trajectory segments of $N - K$.
\end{itemize}

Height functions make it possible to introduce ``cylindrical coordinates'' on $N$. Two coordinate patches $N - N^-$ and $N - N^+$ are needed, and coordinates are defined as \[N - N^{-} \ni p \longmapsto (\pi^i(p),u(p)) \in N^i \times [-1,1]\] in the first patch and analogously in the second, using $\pi^o$ instead. For tangency points the height is not well defined but in practice this will not be a problem because such a point coincides with its own projection $\pi^i(p)$ and so its first cylindrical coordinate determines the point completely anyway. 

We now show how to construct height functions. For each $p \in N$ let $\ell(p) := \nicefrac{1}{2} (-t^i(p) + t^o(p)) \in [0,+\infty]$. This is half the time-length of the trajectory segment of $p$ in $N$, so in particular it is constant over it. Write $N$ as the union of its closed subsets \[N^{\geq 0} := \{p \in N : -t^i(p) \leq \ell(p)\} \ \ \text{ and } \ \ N^{\leq 0} := \{p \in N : -t^i(p) \geq \ell(p)\}.\] These are the points in $N$ that are in the first or second half of their trajectory segment in $N$. Observe that $N^i \subseteq N^{\geq 0}$, $N^o \subseteq N^{\leq 0}$, and $K \subseteq N^{\geq 0} \cap N^{\leq 0}$.

Define a mapping $u^{\geq 0} : N^{\geq 0} \longrightarrow \mathbb{R}$ by \[u^{\geq 0}(p) := 1-\frac{\arctan (-t^i(p))}{\arctan \ell(p)}.\] This is not defined where $\ell = 0$, which is precisely on the tangency curves of $N$.

Consider a point $p \in N^i$, not on a tangency curve, and follow its trajectory in $N$. The map $-t^i$ starts from zero and increases strictly while $\ell$ remains constant, and so $u^{\geq 0}$ starts at $+1$ and decreases strictly. If $p$ does not belong to $N^+$ (so that it eventually exits $N$) then it will reach the half-time point of its trajectory. There $-t^i = \ell$ and so $u^{\geq 0} = 0$. If it belongs to $N^+$ then $-t^i(p)$ will go to $+\infty$ and $\ell(p) = +\infty$ as well, so in the limit $u^{\geq 0}$ will converge to $0$ too. Notice that $u^{\geq 0}|_K =0$ as well because both $-t^i$ and $\ell$ are infinite. Thus $u^{\geq 0}$ satisfies the required properties on $N^{\geq 0}$. Defining $u^{\leq 0}$ in an analogous manner on $N^{\leq 0}$ and observing that both match on the intersection (both are zero) we can take their union to obtain $u$.

\section{Coloured manifolds} \label{sec:coloured}

We abandon dynamics temporarily and introduce a topological abstraction of the dynamical information carried by the boundary of an isolating block:

\begin{definition} \label{defn:coloured} A coloured manifold $N \subseteq \mathbb{R}^3$ is a tame compact $3$--manifold whose boundary $\partial N$ is decomposed as the union of two compact $2$--manifolds $P$ and $Q$ (possibly empty) which meet precisely at their boundary: $P \cap Q = \partial P = \partial Q$.
\end{definition}

Clearly an isolating block is a coloured manifold with $N^i$ and $N^o$ playing the roles of $P$ and $Q$. Following the convention set out earlier for isolating blocks, we colour $P$ white and $Q$ dark gray in the drawings. Notice that in any coloured manifold the set $P \cap Q$, being the boundary of the compact $2$--manifold $P$ (and also of $Q$), is a disjoint union of simple closed curves. We call these the $t$--curves of the coloured manifold. We may orient $\partial N$ as the boundary of $N$ and then $P$ inherits this orientation; we then orient the $t$--curves as the boundary of $P$.

Recall from geometric topology that an elementary way of enlarging a $3$--manifold is by attaching handles onto it. The adaptation to our coloured context is described now.
\medskip

\underline{\itshape Attaching $1$--handles}. Let $\mathbb{D}^1$ be the unit interval $[-1,1]$ and $\mathbb{D}^2$ the closed unit disk in $\mathbb{R}^2$. A $1$--handle $H^{(1)}$ is, topologically, the cartesian product $\mathbb{D}^1 \times \mathbb{D}^2$. Its boundary contains the set $(\partial \mathbb{D}^1) \times \mathbb{D}^2$, which is a disjoint union of the two attaching disks $\{-1\} \times \mathbb{D}^2$ and $\{1\} \times \mathbb{D}^2$. The set $\mathbb{D}^1 \times \mathbb{D}^2$ is given a colouring as shown at the top of Figure \ref{fig:1-handle}.(a).

In order to attach $H^{(1)}$ onto a coloured manifold $N$ we proceed as follows. Select two $t$--curves $\tau_1$ and $\tau_2$ in $\partial N$; they do not need to be different. Then pick:
\begin{itemize}
	\item points $p_1 \in \tau_1$ and $p_2 \in \tau_2$ (if $\tau_1$ and $\tau_2$ happen to be the same curve, choose $p_1 \neq p_2$),
	\item two small disjoint closed disks $D_1, D_2 \subseteq \partial N$ centered at $p_1$ and $p_2$ that lay half in $P$ and half in $Q$, as shown in Figure \ref{fig:1-handle}.(a),
	\item two homeomorphisms $h_1 : \{-1\} \times \mathbb{D}^2 \longrightarrow D_1$ and $h_2 : \{1\} \times \mathbb{D}^2 \longrightarrow D_2$ which preserve the colouring.
\end{itemize}

Let $N' := N \cup_{h_1,h_2} (\mathbb{D}^1 \times \mathbb{D}^2)$ be the space obtained by attaching the $1$--handle $\mathbb{D}^1 \times \mathbb{D}^2$ onto $N$ via the homeomorphisms $h_1$ and $h_2$. That is, $N'$ is the disjoint union of $N$ and $\mathbb{D}^1 \times \mathbb{D}^2$ quotiented by the identifications \[\{-1\} \times \mathbb{D}^2 \ni (-1,x) \sim h_1(x) \in D_1 \ \ \text{ and } \ \ \{1\} \times \mathbb{D}^2 \ni (1,x) \sim h_2(x) \in D_2.\] This is a $3$--manifold. We shall identify the $1$--handle $H^{(1)}$ with the image of $\mathbb{D}^1 \times \mathbb{D}^2$ in $N'$ and also $N$ with its image in $N'$, so that we can write $N' = N \cup H^{(1)}$. The manifold $N'$ can be given an obvious colouring $\partial N' = P' \cup Q'$ as shown in Figure \ref{fig:1-handle}.(b). For instance, $Q'$ is obtained by removing from $Q$ the gray halves of the attaching disks and then adding the gray strip that runs along the $1$--handle.

The construction of $N'$ involves a number of choices. It is clear that, up to a colour preserving homeomorphism, moving the $p_i$ along the $\tau_i$ or changing the disks $D_i$ will have no effect on $N'$. Also, since any two orientation preserving homeomorphisms of a disk are isotopic, as long as $N'$ is orientable (which will always be our case since we work within $\mathbb{R}^3$) the coloured manifold $N'$ does not depend on the $h_i$ either. This is the ``coloured version'' of a standard result in geometric topology (see for instance \cite[Lemma 6.1, p. 75]{rourkesanderson1}). Thus $N'$ is uniquely determined (up to a colour preserving homeomorphism) just by stating onto which $t$--curves $\tau_1$ and $\tau_2$ the handle should be pasted.

\begin{figure}[h!]
\subfigure[The setup]{
\scalebox{0.9}{
\begin{pspicture}(0,-0.2)(6.5,6)
	\rput[bl](0,0){\scalebox{0.5}{\includegraphics{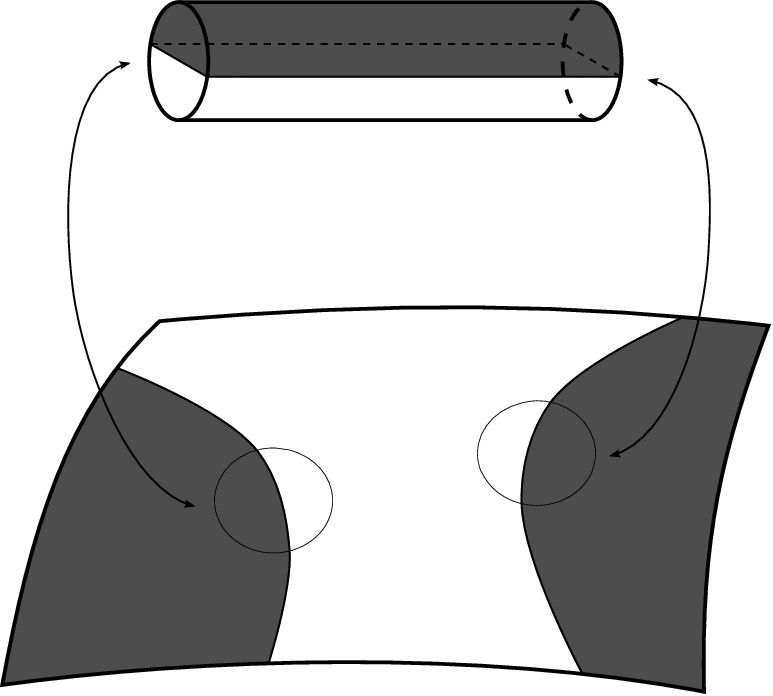}}}
	\rput[t](3.25,4.7){$\mathbb{D}^1 \times \mathbb{D}^2$}
	\rput[l](2.8,1.2){$D_1$} \rput[r](4.3,1.5){$D_2$}
	\rput[r](0.5,4){$h_1$} \rput[l](6.1,4){$h_2$}
	\rput[l](1.4,2.7){$\tau_1$} \rput(5,3){$\tau_2$}
	\rput(6.3,0){$N$}
\end{pspicture}}}
\quad
\subfigure[The handle attached]{
\scalebox{0.9}{\begin{pspicture}(0,-0.2)(6.5,5)
	\rput[bl](0,0){\scalebox{0.5}{\includegraphics{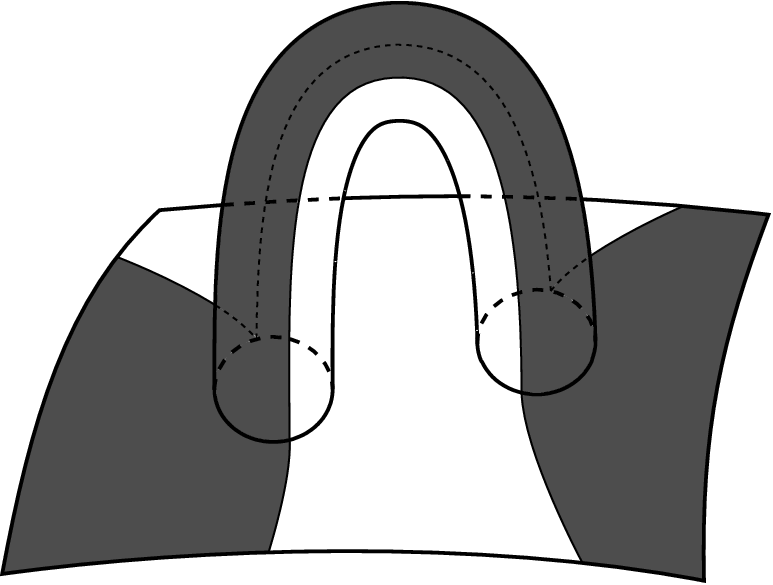}}}
	\rput[bl](4.5,4.5){$H^{(1)}$}
	\rput(6.3,0){$N'$}
\end{pspicture}}}
\caption{Attaching a $1$--handle \label{fig:1-handle}}
\end{figure}
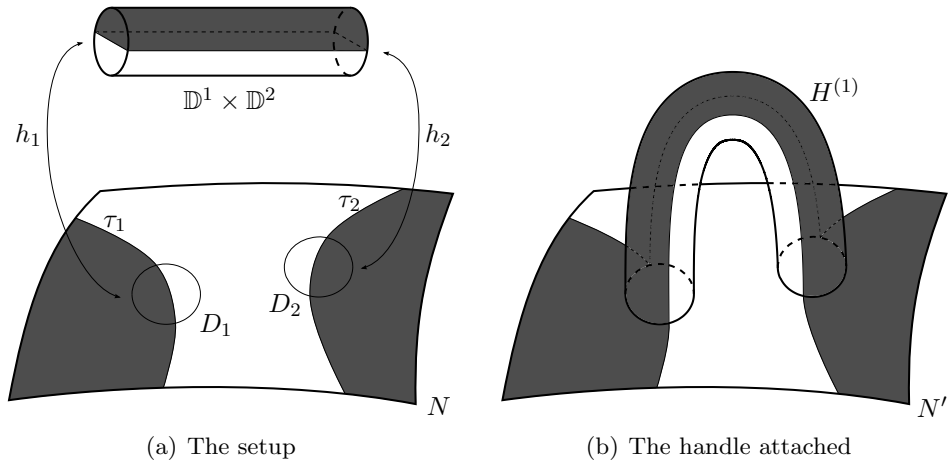

When a $1$--handle is pasted onto two different $t$--curves ($\tau_1 \neq \tau_2$) the two coalesce and give rise to a single $t$--curve of the new coloured manifold $N'$. A $1$--handle can also be pasted onto a single $t$--curve ($\tau_1 = \tau_2$) of $N$, and then it may split it into two $t$--curves or produce a single one depending on whether the $1$--handle is ``twisted''. However, since we only consider an orientable $N'$ the latter case cannot happen for the following reason. Adding a $1$--handle onto $N$ has the effect of pasting a black ribbon onto $Q$ and changing its Euler characteristic to $\chi(Q') = \chi(Q)-1$. Since the Euler characteristic of an orientable surface has the same parity as its number of boundary components, adding a $1$--handle always changes the number of $t$--curves.

We record the following as a remark for later reference:

\begin{remark} \label{rem:thom} Suppose $N' := N \cup H^{(1)}$ is the result of attaching a $1$--handle onto $N$. If the handle joins two different components of $N$ we have $H_1(N') = H_1(N)$. If it is pasted onto a single component of $N$ then $H_1(N') = H_1(N) \oplus \langle h \rangle$ where $h$ is any simple closed curve going once along the handle and then closing up in $N$. Here $\langle h \rangle$ means the free Abelian group generated by $h$.

For a $t$--curve $\tau$ of $N$ we denote by $[\tau]$ its homology class in $H_1(N)$, and similarly for the $t$--curves of $N'$. We discuss the effect of pasting a $1$--handle on the homology classes represented by the $t$--curves. Orient all $t$--curves as the boundary of $P$. Then if the $1$--handle joins two different $t$--curves $\tau_1$ and $\tau_2$ of $N$ to produce a single $\tau'$ one has $[\tau'] = [\tau_1] + [\tau_2]$ because $\tau' - \tau_1 - \tau_2$ is the boundary of the white rectangle along the $1$--handle. If the $1$--handle is pasted onto a single $\tau$ it splits it into $\tau'_1$ and $\tau'_2$ and now $[\tau'_1] + [\tau'_2] = [\tau]$ in $H_1(N')$ for the same reason. Moreover, since $\tau'_1$ (say) runs just once along the handle we may take it to be the new generator $h$ of the homology of $N'$ and so the new $t$--curves are $[\tau'_1] = h$ and $[\tau'_2] = [\tau] - h$.
\end{remark}

\medskip

\underline{\itshape Attaching $2$--handles}. A $2$--handle $H^{(2)}$ is, topologically, the cartesian product $\mathbb{D}^2 \times \mathbb{D}^1$. It gets attached onto a manifold along an annulus, namely $(\partial \mathbb{D}^2) \times \mathbb{D}^1 \subseteq \partial H^{(2)}$. The set $\mathbb{D}^2 \times \mathbb{D}^1$ is given a colouring as shown at the top of Figure \ref{fig:2-handle}.(a).

To attach the $2$--handle onto a coloured manifold $N$ we select a $t$-- curve $\tau$ in $\partial N$ and then pick:
\begin{itemize}
	\item a thin closed annulus $A$ along $\tau$,
	\item a homeomorphism $h : (\partial \mathbb{D}^2) \times \mathbb{D}^1 \longrightarrow A$ which preserves the colouring.
\end{itemize}

Again we use the shorthand $N' = N \cup H^{(2)}$ to denote the attaching space $N' := N \cup_{h} (\mathbb{D}^2 \times \mathbb{D}^1)$ which receives the colouring shown in Figure \ref{fig:2-handle}.(b) (notice that the inner face of the $2$--handle is still painted white although this is not visible in the drawing). As in the case of $1$--handles the choice of $A$ and $h$ is irrelevant as far as the coloured manifold $N'$ is concerned; only the attaching curve $\tau$ is relevant.

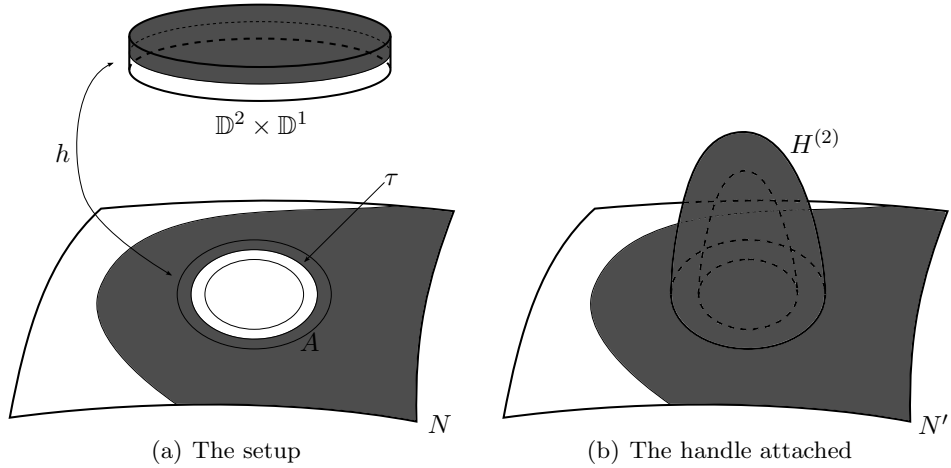
\begin{figure}[h!]
\subfigure[The setup]{
\scalebox{0.9}{\begin{pspicture}(0,0)(6.5,6.2)
	\rput[bl](0,0){\scalebox{0.5}{\includegraphics{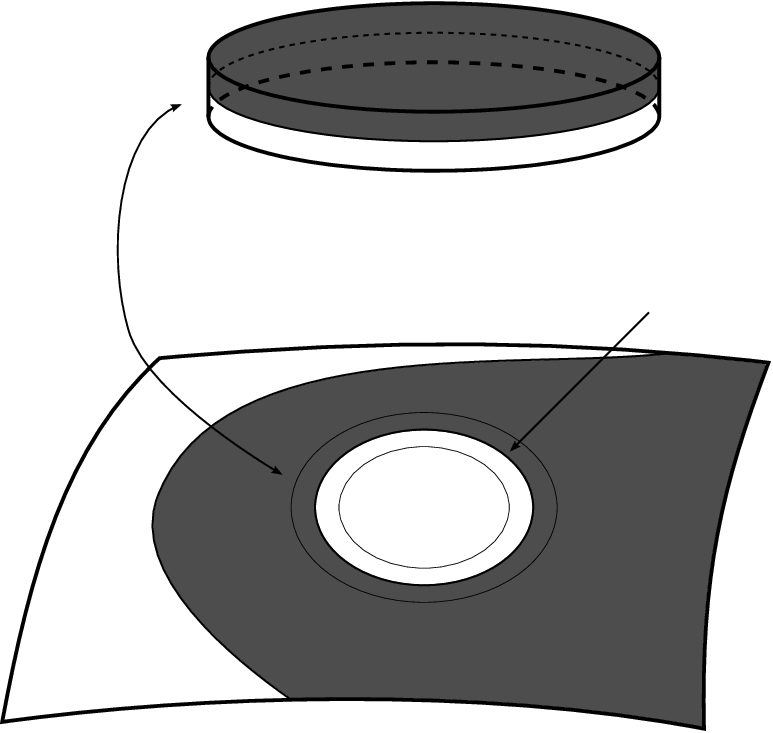}}}
	\rput[r](0.9,4){$h$} \rput(5.6,3.6){$\tau$}
	\rput[t](3.7,4.6){$\mathbb{D}^2 \times \mathbb{D}^1$}
	\rput(4.4,1.2){$A$}
	\rput(6.3,0){$N$}
\end{pspicture}}}
\quad
\subfigure[The handle attached]{
\scalebox{0.9}{\begin{pspicture}(0,0)(6.5,4.4)
	\rput[bl](0,0){\scalebox{0.5}{\includegraphics{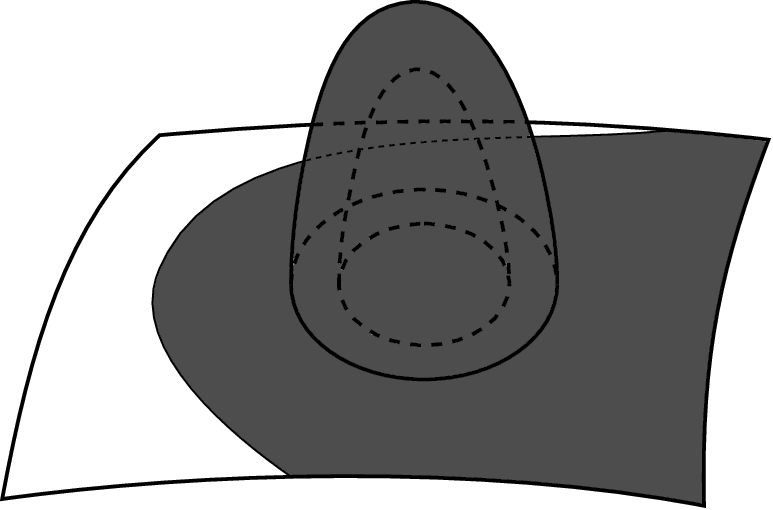}}}
	\rput[bl](4.2,4){$H^{(2)}$}
	\rput(6.3,0){$N'$}
\end{pspicture}}}
\caption{Attaching a $2$--handle \label{fig:2-handle}}
\end{figure}

\medskip

\section{A handle theorem} \label{sec:asas}

In this section we prove the following:

\begin{theorem} \label{teo:asas} Let $K$ be an isolated invariant set. Assume $\check{H}^k(K)$ is finitely generated for $k = 0,1$. Let $N$ be a connected isolating block for $K$. Then there exists a regular isolating block $B \subseteq {\rm Int}\ N$ for $K$ such that $N$ can be obtained by suitably attaching coloured $1$-- and $2$--handles (in that order) onto $B$.
\end{theorem}

The proof starts from a ``shaving and making convex to the flow'' technique which is a blend of ideas from \cite{churchill1} and \cite{conleyeaston1}. We first illustrate it with a simple example. Referring to Figure \ref{fig:sketch1}.(a), consider a flow in $\mathbb{R}^3$ which has a circle $K$ as an invariant set (comprised of fixed points) and trajectories otherwise flow vertically downwards. Figure \ref{fig:sketch1}.(b) shows a ball $N$ which is an isolating block for $K$. The entry and exit sets $N^i$ and $N^o$ are the upper and lower hemispheres of $\partial N$, respectively. Figure \ref{fig:sketch1}.(b) also shows $n^+$, which is a circle in the upper hemisphere.

\begin{figure}[h!]
\null\hfill
\subfigure[A flow in $\mathbb{R}^3$]{
\begin{pspicture}(0,0)(3.5,6)
	\rput[bl](0,0){\scalebox{0.9}{\includegraphics{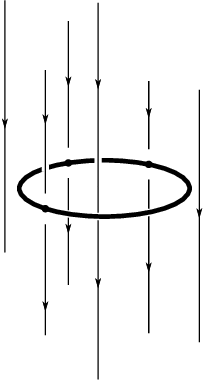}}}
	\rput[bl](2.5,2.2){$K$}
\end{pspicture}}
\hfill
\subfigure[An isolating block for $K$]{
\begin{pspicture}(0,0)(5,6)
	\rput[bl](0,0){\scalebox{0.9}{\includegraphics{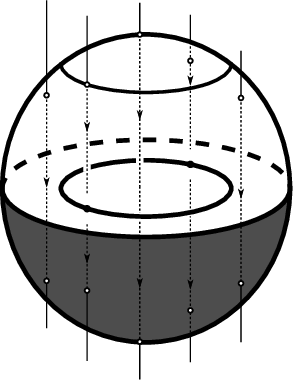}}}
	\rput[bl](4.6,2.2){$N$}
\end{pspicture}}
\hfill\null
\caption{ \label{fig:sketch1}}
\end{figure}

Now we choose a small closed disk $E \subseteq {\rm Int}\ N^i$ disjoint from $n^+$. This is shown in light gray in Figure \ref{fig:sketch2}.({\it a}\/). If we follow the trajectory of the disk under the flow it traces a solid cylinder inside $N$, disjoint from $K$. We remove this from $N$; see Figure \ref{fig:sketch2}.({\it b}\/). The resulting object $B_0$ is not an isolating block for $K$ because the lateral face of the cylindrical hole consists of points which are neither transverse entry or exit points or exterior tangency points. We easily fix this by tapering the cylinder to make it convex to the flow as in Figure \ref{fig:sketch2}.({\it c}\/). This produces an isolating block $B$. Looking at the portion of $\partial B$ contained in the now hourglass-shaped ``well'' going through the ball, the flow is externally tangent to $B$ at the ``waist'' and otherwise enters and exits $B$ transversally through the upper and lower halves of the well, respectively.

Our contribution to this construction is the following. Observe that $N$ can be recovered from $B$ by plugging a (very thick) $2$--handle back into the hourglass-shaped well we have drilled. The wall of the well is an annulus along the tangency curve at its waist, with its upper half is painted white (as part of the entry set) and the lower half grey (as part of the exit set). Thus, as a coloured manifold, $N$ is the result of pasting a coloured $2$--handle onto $B$. Finally, notice that $B$ is a solid torus and is a regular isolating block for $K$. We will prove that given any initial $N$, a regular isolating block can be constructed using this technique and an appropriate sequence of disk removals (as in the example) or cuts along strips (to be described later). The reverse of each of these is the attaching of a coloured $2$--handle or $1$--handle, respectively, and this is essentially the handle theorem.

\begin{figure}[h!]
\null\hfill
\subfigure[]{
\begin{pspicture}(0,0)(5,5)
	\rput[bl](0,0){\scalebox{0.9}{\includegraphics{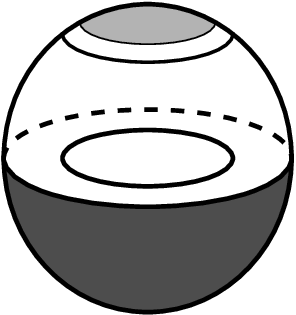}}}
	\rput[bl](2,4.3){$E$}
\end{pspicture}}
\hfill
\subfigure[]{
\begin{pspicture}(0,0)(5,5)
	\rput[bl](0,0){\scalebox{0.9}{\includegraphics{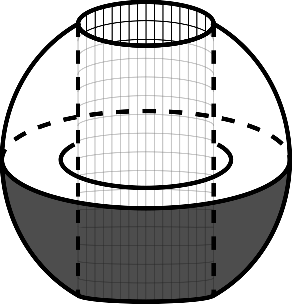}}}
 \rput[bl](4.3,0.8){$B_0$}
\end{pspicture}}
\hfill
\subfigure[]{
\begin{pspicture}(0,0)(5,5)
	\rput[bl](0,0){\scalebox{0.9}{\includegraphics{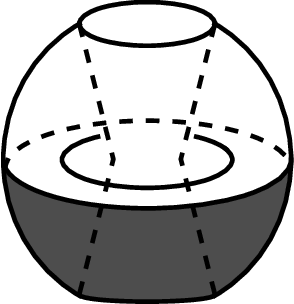}}}
	\rput[bl](4.3,0.8){$B$}
\end{pspicture}}
\hfill\null
\caption{ \label{fig:sketch2}}
\end{figure}

\subsection{} Let us discuss the construction of the example in general. We start with an arbitrary isolating block $N$ and consider a small disk $E \subseteq {\rm Int}\ N^i$ disjoint from $n^+$. Set \[C := \{p \in N : \pi^i(p) \in E\}\] and let $B_0$ be the (closure of the) remaining portion of $B$; namely \[B_0 := \{p \in N : \pi^i(p) \not\in {\rm Int}\ E \} \cup N^-.\] Both are compact sets whose union is $B$. $B_0$ is a neighbourhood of $K$ whereas $C$ is disjoint from $N^+ \cup N^-$. Notice also that $C$ is homeomorphic to $E \times [-1,1]$ via cylindrical coordinates. Identifying $E \times [-1,1] = \mathbb{D}^2 \times \mathbb{D}^1$, one has \[B_0 \cap C = \{p \in N : \pi^i(p) \in \partial E\} \cong (\partial \mathbb{D}^2) \times \mathbb{D}^1,\] so $N$ is (topologically) the result of pasting a $2$--handle onto $B_0$. We still need to make $B$ convex to the flow.

Let $E_1$ and $E_3$ be closed disks, slightly smaller and slightly bigger than $E$ respectively, so that $E_1 \subseteq E \subseteq E_3$ can be identified with three nested disks of radii $1,2,3$ in the $xy$ plane. Let $\alpha : E_3 \longrightarrow [0,3]$ be the radius function provided by this identification. 

Let \[B := B_0 \cup \{p \in C : \alpha(\pi^i(p)) \geq 1+|u(p)|\}\] and \[H^{(2)} := \{p \in C : \alpha(\pi^i(p)) \leq 1+|u(p)|\},\] so that evidently $N = B \cup H^{(2)}$. A rough picture is shown in Figure \ref{fig:convexflow}. Panel (b) shows $C$ ``straightened up'' so that the height function $u$ really accords to the height in the drawing. The disk $E$ appears as an interval, and $\alpha$ measures distance from the center of the interval.

\begin{figure}[h]
\null\hfill
\subfigure[]{
	\begin{pspicture}(0,0)(4,4)
	\rput[lb](0,0){\scalebox{0.5}{\includegraphics{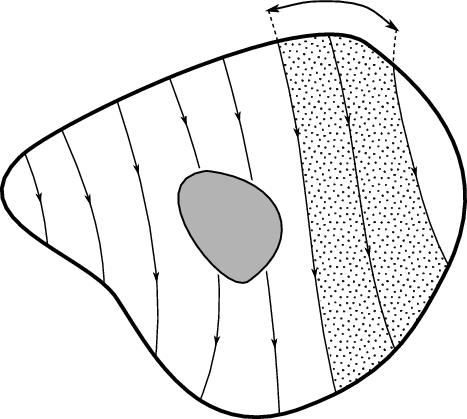}}}
	\rput(1.95,1.7){$K$} \rput(2.8,3.8){$E$} \rput(3.4,1.3){$C$} \rput[b](2,0.2){$B_0$}
	\end{pspicture}}
\hfill
\subfigure[]{
	\begin{pspicture}(0,0)(4,4)
	\rput[lb](0,0){\scalebox{0.5}{\includegraphics{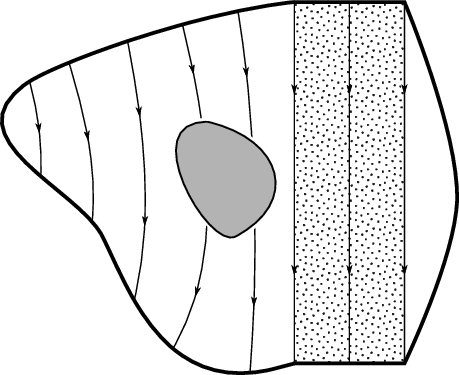}}}
	\rput(1.9,1.7){$K$} \rput(3.2,1.7){$C$} \rput[b](1.8,0.2){$B_0$}
	\end{pspicture}}
\hfill
\subfigure[]{
	\begin{pspicture}(0,0)(5,4)
	\rput[lb](0,0){\scalebox{0.5}{\includegraphics{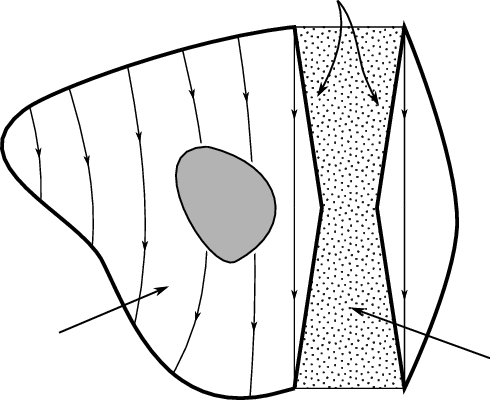}}}
	\rput(1.9,1.7){$K$} 
	\rput[b](2.9,3.4){$\alpha = 1+|u|$}
	\rput[bl](0.1,0.2){$B$} \rput[bl](4.3,0.2){$H^{(2)}$}
	\end{pspicture}}
\hfill\null
\caption{ \label{fig:convexflow}}
\end{figure}

To interpret $B$ geometrically let $p \in E$ move along a trajectory segment from top to bottom. The projection $\pi^i(p)$ and hence also $\alpha(\pi^i(p))$ remain constant, while the height $u(p)$ goes from $+1$ to $-1$. If $\alpha<1$ (i.e. the trajectory segment begins at a point in the interior of the smaller disk $E_1$) then the segment does not intersect $B$ because the condition $\alpha \geq 1 + |u|$ is never met. If $\alpha = 1$ the segment intersects $B$ exactly once at height $u = 0$, so it is externally tangent to $B$ at that point. If $1 < \alpha < 2$ the segment intersects $B$ in a whole interval between heights $\pm (\alpha-1)$, entering $B$ transversally at height $u = \alpha-1$ and exiting it transversally at height $u = -(\alpha-1)$. Transversality is ensured because $u$ is strictly decreasing. Finally if $\alpha = 2$ the segment is entirely contained in $B$.

\begin{proposition} \label{prop:disk} The following hold:
\begin{itemize}
    \item[(i)] $B$ is an isolating block for $K$.
    \item[(ii)] $B^+ = N^+$ and $b^+ = n^+$. Moreover, the pair $(B^i,b^+)$ is homeomorphic to $(N^i - {\rm int}\ E, n^+)$.
    \item[(iii)] $N$ can be recovered from $B$, as a coloured manifold, by attaching a $2$--handle onto it.
\end{itemize}
\end{proposition}
\begin{proof} (i) Enlarge $C$ slightly to the open set $U := \{p \in N : \pi^i(p) \in E_3\}$. Via cylindrical coordinates, and recalling the definition of $\alpha$ as a radius, $U$ is homeomorphic to the cylinder $\{(x,y,z) \in \mathbb{R}^3 : \sqrt{x^2+y^2} < 3 , |z| \leq 1\}$. Via this homeomorphism $B \cap U$ goes to $\{(x,y,z) \in \mathbb{R}^3 : 3 > \sqrt{x^2+y^2} \geq 1 + |z|\}$. $B$ is the union of its open subsets $N - C$ and $B \cap U$. Since $N$ is a $3$--manifold by definition, so is its open subset $N - C$; $B \cap U$ is also easily seen to be a $3$--manifold with the above explicit description up to homeomorphism. Thus $B$ is a $3$--manifold whose boundary is the union of the boundaries of $N-C$ and $B \cap U$ which a quick check shows is \[\partial B = (\partial N \cap B_0) \cup \{p \in C : \alpha(\pi^i(p)) = 1 + |u(p)|\}.\] 

It follows from the geometric analysis before the proposition that:

(i) $\tau := \{p \in C : \alpha(\pi^i(p)) = 1 \text{ and } u(p) = 0\}$ is a tangency curve of $B$. The remaining tangency curves are those of $N$.

(ii) The entry set of $B$ is \[B^i = (N^i \cap B_0) \cup \{p \in C : \alpha(\pi^i(p)) = 1+|u(p)| \text { with } u(p) \geq 0\}\] and similarly the exit set is \[B^o = (N^o \cap B_0) \cup \{p \in C : \alpha(\pi^i(p)) = 1+|u(p)| \text{ with } u(p) \leq 0\};\] both the entry and exit happens transversally away from the tangency curves.

To finish proving that $B$ is an isolating block we need to show that it is locally tame at its tangency points as explained in Subsection \ref{sub:tame}. The only new tangency points are those in $\tau$. However, for those points we have a homeomorphism of $(U,B \cap U)$ with the explicit subset of $\mathbb{R}^3$ described at the beginning of the proof. In that model it is clear that $B$ is locally tame at points in $\tau$.

(ii) Let us focus on the set $A^{\geq 0} := \{p \in C : \alpha(\pi^i(p)) = 1 + |u(p)| \text{ with } u(p) \geq 0\}$ that appears in $B^i$. Each trajectory segment in $C$ intersects $A^{\geq 0}$ just once (for $1 \leq \alpha \leq 2$) or not at all (for $\alpha < 1$) so the projection $\pi^i$ establishes a homeomorphism from $A^{\geq 0}$ onto the annulus $2 \geq \alpha \geq 1$ and (since $\pi^i$ leaves points in $N^i$ unaltered) from  $B^i$ onto $N^i - \{\alpha < 1\}$, i.e. onto $N^i$ with the interior of small disk $E_1$. Notice also that $b^+ = n^+$ and $\pi^i$ leaves this set unaltered. Thus $(B^i,b^+)$ is homeomorphic to $(N^i - {\rm int}\ E_1, n^+)$. Clearly there is a homeomorphism of the latter pair onto $(N^i - {\rm int}\ E, n^+)$.

(iii) Defining $A^{\leq 0}$ in a manner completely analogous to $A^{\geq 0}$, we see that it is also homeomorphic via $\pi^i$ to the annulus $1 \leq \alpha \leq 2$. Since $H^{(2)}$ is entirely contained in $U$ we can use the homeomorphism from $U$ to a subset of $\mathbb{R}^3$ given above to check that $H^{(2)}$ is indeed a $2$--handle with the appropriate colouring. Finally, observe that $B \cap H^{(2)} = \{p \in C : \alpha(\pi^i(p)) = 1 + |u(p)|\}$. This is exactly $A^{\geq 0} \cup A^{\leq 0}$, so it is an annulus along the tangency curve $\tau$. Thus $N$ can be recovered from $B$ as a coloured manifold by attaching the coloured $2$--handle $H^{(2)}$ along the tangency curve $\tau$.
\end{proof}

\subsection{} In addition to removing disks from $N^i$ we will also need to cut it along strips. Let again $N$ be an isolating block and picture $n^+$ in the interior of the surface $N^i$. Suppose $\gamma \subseteq N^i$ is a simple arc which meets $\partial N^i$ (i.e. some tangency curve(s)) precisely at its endpoints. We pick a thin closed strip $E$ along $\gamma$, thin enough that it does not intersect $n^+$ either, and remove its interior from $N^i$ thereby cutting $N^i$ along the strip. See Figure \ref{fig:surf_remove}.

\begin{figure}[h]
\begin{pspicture}(0,0)(11,4.5)
\rput[bl](0,0){\scalebox{0.75}{\includegraphics{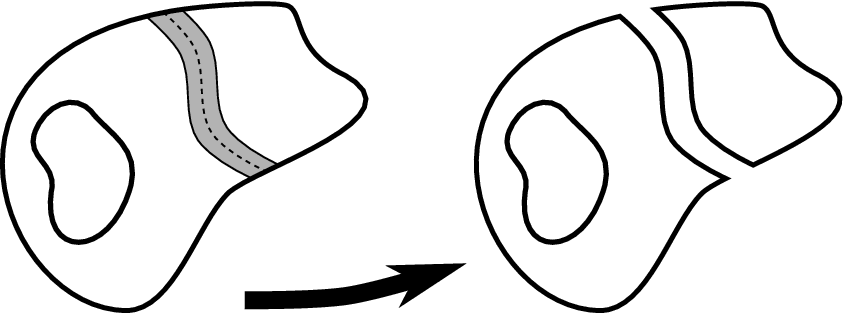}}}
\rput[bl](1,3.8){$N^i$} 
\rput[bl](2.9,2.7){$E$} 
\end{pspicture}
\caption{\label{fig:surf_remove}}
\end{figure}

The whole construction of the preceding section can be repeated to extend this to all of $N$. We define $C$ and $B_0$ in the exact same way. Now cylindrical coordinates do not directly provide a homeomorphism between $C$ and $E \times [-1,1]$ because of the degeneracy of the height function at $E \cap \partial N^i$ which is comprised of tangency points. To account for this we only need to collapse each $p \times [-1,1] \subseteq  E \times [-1,1]$ to a single point for those $p \in E \cap \partial N^i$. The resulting quotient is homeomorphic to $\mathbb{D}^1 \times \mathbb{D}^2$ and under this identification $B_0 \cap C \cong \{\pm 1\} \times \mathbb{D}^2$. (These two disjoint disks arise from the long sides of $E$). Hence, topologically $N = B_0 \cup C$ is the result of attaching a $1$--handle onto $B_0$. One can make $B_0$ convex to the flow to obtain an isolating block $B$ and then $N$ turns out to be $B$ with a coloured $1$--handle pasted onto it. As before, $(B^i,b^+) \cong (N^i - 
{\rm int}\ E,n^+)$. We omit further details and simply state the following:

\begin{proposition} \label{prop:strip} The following hold:
\begin{itemize}
    \item[(i)] $B$ is an isolating block for $K$.
    \item[(ii)] $B^+ = N^+$ and $b^+ = n^+$. Moreover, the pair $(B^i,b^+)$ is homeomorphic to $(N^i - {\rm int}\ E, n^+)$.
    \item[(iii)] $N$ can be recovered from $B$, as a coloured manifold, by attaching a $1$--handle onto it.
\end{itemize}
\end{proposition}

\subsection{} The last ingredient to prove Theorem \ref{teo:asas} is the following topological result:

\begin{theorem} \label{teo:reg_surf_1} Let $L$ be a compact subset of the interior of a compact surface $S$. Assume $\check{H}^k(L)$ is finitely generated for $k = 0,1$. Then by suitably removing disks, cutting along strips, and finally discarding unnecesary components (in this order) it is possible to trim $S$ down to a compact surface $P$ that still contains $L$ in its interior and such that $L \subseteq P$ induces isomorphisms in \v{C}ech cohomology.

Letting $r_i = {\rm rk}\ \check{H}^{2-i}(S,L)$, at most $r_0$, $r_1$, and $r_2$ operations of each type are needed.
\end{theorem}

The existence of such a $P$ is essentially well known (see for example the frame theorem of Moise \cite[Theorem 6, p. 72]{moise2}). However, we have not found any reference that includes the explicit information on how to construct $P$. Since this is essential to us, we provide a proof of the theorem in Appendix \ref{app:frame}. Now we prove the handle theorem:

\begin{proof}[Proof of Theorem \ref{teo:asas}] Since $N$ and $N^i$ are compact manifolds they have a finitely generated cohomology. The assumption on the cohomology of $K$ then implies that $n^+$ has a finitely generated cohomology in degrees $0$ and $1$ using the isomorphism $\check{H}^*(N^i,n^+) = \check{H}^k(N,K)$ provided by Proposition \ref{prop:coh}.(iii) and the long exact sequences of the pairs. Moreover $\check{H}^0(N^i,n^+) = \check{H}^0(N,K) = 0$ because $N$ is connected. Thus according to Theorem \ref{teo:reg_surf_1} we may perform a sequence of disks removals and cuts along strips to obtain a neighbourhood $P$ of $n^+$ such that $n^+ \subseteq P$ induces isomorphisms in \v{C}ech cohomology. Let us write $P_0 = N^i, P_1, \ldots, P$ for the sequence of surfaces this produces. This leads to a sequence of isolating blocks $N_0 := N, B_1, \ldots, B$ as follows. Consider for instance the first step in the sequence, going from $P_0 = N^i$ to $P_1$ via, say, a disk removal. Using the technique explained above and summarized in Proposition \ref{prop:disk} we may extend this to all of $N$, thus writing $N = B_1 \cup H^{(2)}$ where $B_1$ is a new isolating block with $(B_1^i,n_1^+) \cong (P_1,n^+)$. The next disk removal should be performed in $P_1$, but via this homeomorphism we may just as well perform it in $B_1^i$ instead. Thus again we may extend it to all of $B_1$, and so on. Inductively we end up with an isolating block $B$ with $(B^i,b^+) \cong (P,n^+)$. Then by Proposition \ref{prop:coh} we have \[\check{H}^*(B,K) = \check{H}^*(B^i,b^+) = \check{H}^*(P,n^+) = 0\] and so $B$ is a regular isolating block for $K$. Notice that at each step $B_i$ can be recovered from $B_{i+1}$ via a coloured handle attachment; first $2$--handles (corresponding to disk removals) and then $1$--handles (corresponding to cut-along-strip operations). Hence the initial $N$ can then be recovered from the final $B$ by attaching handles in the reverse order.
\end{proof}

We will need a slightly extended version of the handle theorem where only partial information about the cohomology of $K$ is known:

\begin{theorem} \label{teo:asasgen} Let $K$ be an isolated invariant set which is the union of two disjoint compacta $K_0$ and $K_1$ (these are then isolated invariant sets themselves). Assume $\check{H}^k(K_0)$ is finite dimensional for $k = 0,1$. Let $N$ be a connected isolating block for $K$. Then there exist a regular isolating block $B_0$ for $K_0$ and an isolating block $B_1$ for $K_1$, both disjoint and contained in ${\rm int}\ N$, such that $N$ can be obtained by suitably attaching coloured $1$-- and $2$--handles onto $B_0 \cup B_1$.
\end{theorem}

\begin{proof} We first show that $N^+$ decomposes as a disjoint union of two compact sets $N^+_i$, each each characterized because its points are attracted towards $K_i$. Pick disjoint closed neighbourhoods $U_i$ of $K_i$ in $N^+$. Then $U_0 \cup U_1$ is a neighbourhood of $K$ in $N^+$ and so there exists $t_0$ such that $N^+ \cdot [t_0,+\infty) \subseteq U$. Define \[N^+_i := \{p \in N^+ : p \cdot t_0 \in U_i\}.\] Clearly these are disjoint closed sets (because so are the $U_i$) whose union is $N^+$. For each $p \in N^+$ the semiorbit $p \cdot [t_0,+\infty) \subseteq U = U_0 \uplus U_1$ is connected and hence contained entirely in the same $U_i$ that contains $p \cdot t_0$. Therefore each $N_i^+$ is positively invariant and attracted by $K_i$.

Decompose $n^+$ as the disjoint union of the compact sets $n^+_i := N^+_i \cap \partial N$ for $i = 0,1$. Now we have that \[(N^+/n^+,*) = (N^+_0/n_0^+,*) \vee (N^+_1/n_1^+,*)\] which implies $\check{H}^*(N^+/n^+,*) = \check{H}^*(N_0^+/n_0^+,*) \oplus \check{H}^*(N_1^+/n_1^+,*)$ and, since the cohomology of $N^+/n^+$ is finitely generated by Proposition \ref{prop:coh}.(ii), the same is true of the cohomology of the $N^+_i/n^+_i$. Observe that the cohomology of $N^+_0$ is finitely generated as well because it coincides with that of $K_0$ by Proposition \ref{prop:coh}.(i). Then using $\check{H}^*(N^+_0/n^+_0,*) = \check{H}^*(N^+_0,n^+_0)$ and the long exact sequence for the pair we have that $n^+_0$ also has a finitely generated \v{C}ech cohomology.

Let $L_1 \subseteq {\rm Int}\ N^i$ be a compact $2$--manifold which is a neighbourhood of $n^+_1$ disjoint from $n^+_0$. Then the compact set $n_0^+ \cup L_1$ has a finitely generated \v{C}ech cohomology and so by Theorem \ref{teo:reg_surf_1} one can perform a sequence of disk removals and cuts along strips on $N^i$ that will produce a surface $P$ such that $n^+_0 \cup L_1 \subseteq P$ induces isomorphisms in \v{C}ech cohomology. We then apply the technique used to prove the handle theorem. At the end we will obtain an isolating block $B$ for $K$ such that $(B^i,b^+) \cong (P,n^+)$.

Let $C$ be a connected component of $B$, which is an isolating block by itself with entry set $C^i$. We prove that $C$ cannot intersect both $K_0$ and $K_1$. Since $C^+ \subseteq B^+ =  N^+ = N^+_0 \uplus N_1^+$ we may write $C^+$ as the disjoint union of the two compact sets $C_0^+ := C^+ \cap N_0^+$ and $C_1^+ := C^+ \cap N_1^+$. These sets intersect $C^i$ in a subset of $n_0^+$ and $n_1^+$ respectively. By construction $C^i$ is a union of components of $P$ and none of these intersects both $n_0^+$ and $n_1^+$ simultaneously. Write $C^i_0$ for the union of the components of $C^i$ which intersect $n^+_0$ and $C^i_1$ for the union of the remaining ones. Then we have that $C^+_0 \cap C^i = C^+_0 \cap C^i_0$ and similarly $C^+_1 \cap C^i = C^+_1 \cap C^i_1$, and so \[C^i \cup C^+ = (C^i_0 \cup C^+_0) \uplus (C^i_1 \cup C^+_1)\] exhibits the connected set $C^i \cup C^+$ as the disjoint union of two compact sets. However, by deforming the connected set $C$ onto $C^i \cup C^+$ in infinite time (as in the proof of Proposition \ref{prop:coh}) we see that $C^i \cup C^+$ must be connected too, and so one of the sets displayed above must be empty. In particular $C^+$ is entirely contained in one of the $N_i^+$ and so the maximal invariant subset of $C$ is a subset of either $K_0$ or $K_1$.

Write $B = B_0 \uplus B_1$ where $B_0$ is the union of those components of $B$ that intersect $K_0$ and $B_1$ is the union of the rest. By construction $B_1$ is disjoint from $K_0$ and by the paragraph above $B_0$ is disjoint from $K_1$. Thus each $B_i$ is an isolating block for $K_i$. It only remains to show that $B_0$ is regular. By the construction of $B$ we have that $b^+ = n^+$, and by the discussion above we know that $B_0$ contains $N_0^+$ and is disjoint form $N_1^+$, so in fact $b_0^+ = n_0^+$. Then by Proposition \ref{prop:coh} we have $H^*(B_0,K_0) = H^*(B_0^i,b_0^+) = H^*(P_0,n_0^+)$ where $P_0$ is the union of the components of $P$ that intersect $n_0^+$. By construcion of $P$ we have $\check{H}^*(P_0,n_0^+) = 0$ and so $B_0$ is a regular isolating block for $K_0$.
\end{proof}

\subsection{} The same woodworking described in this section can be used to prove the existence of a basis of  isolating blocks for any isolated invariant set $K$. 
Churchill \cite{churchill1} proved that every isolated invariant set in a locally compact metric space has a neighbourhood basis of isolating neighbourhoods $N$ with continuous $t^i$ and $t^o$ mappings. We take these as our starting point:

\begin{proposition} Let $N$ be an isolating neighbourhood with $N^i$ and $N^o$ closed. Then the set of transverse entry points is a boundariless $2$--manifold which contains $n^+$ and is open in ${\rm fr}\ N$.
\end{proposition}
\begin{proof} We shall make use of Remark \ref{rem:structure}. First we prove the following:

{\it Claim.} Let $p \in {\rm fr}\ N$ satisfy $p \cdot (0,\epsilon] \subseteq {\rm int}\ N$ for some $\epsilon > 0$. Then there exist $\delta > 0$ and a neighbourhood $V$ of $p$ in ${\rm fr}\ N$ such that every $q\in V$ satisfies $q \cdot (-\delta,0) \cap N = \emptyset$ and $q \cdot (0,\delta) \subseteq {\rm int}\ N$. In particular, each $q \in V$ is a transverse entry point.

{\it Proof.} Let $U$ be a neighbourhood of $p$ in $N$ such that $t^o|_U > \epsilon$. Choose $\delta \in (0,\epsilon)$ and a neighbourhood $V$ of $p$ in ${\rm fr}\ N$ such that $(V \cdot [-\delta,\delta]) \cap N \subseteq U$ and $V \cdot \delta \subseteq {\rm int}\ N$. Let $q \in V$ and consider the trajectory segment $\sigma := q \cdot (-\delta,\delta)$. By construction $q \cdot [0,\delta] \subseteq N$ and $q \cdot \delta \in {\rm int}\ N$ so by (*) we have $q \cdot (0,\delta) \subseteq {\rm int}\ N$. To prove that $q \cdot (-\delta,0) \cap N = \emptyset$, assume $\sigma$ intersects $N$ at some point $q' := q \cdot t$ with $t < 0$; i.e. coming before $q$. Since $q' \in U$ we have $t^o(q') > \epsilon$ and so $q' \cdot [0,\epsilon]$ is entirely contained in $N$. Writing it as $q \cdot [-t,0] \cup q \cdot [0,\epsilon]$ we observe that it intersects ${\rm fr}\ N$ at the interior point $q$, so by (*) the whole segment must be contained in ${\rm fr}\ N$. However, $q \cdot \epsilon \in {\rm int}\ N$. $_{\blacksquare}$
\medskip

Let $O \subseteq N^i$ be the set of transverse entry points. It is open by the preceding claim. It contains $n^+$: since the maximal trajectory segment of every $p \in n^+$ cannot be contained in ${\rm fr}\ N$, Remark \ref{rem:structure} ensures that $p \cdot (0,+\infty) \subseteq {\rm int}\ N$ and then the claim above implies that $p$ is a transverse entry point. Proposition \ref{prop:loc_prod} implies that each such point $p$ has a neighbourhood $V$ such that $V \times (-\delta,\delta) \cong V \cdot (-\delta,\delta)$ via the flow and the latter is an open neighbourhood of $p$ in $\mathbb{R}^3$. In particular $V \times (-\delta,\delta)$ is a $3$--manifold, and this implies that $V$ is a $2$--manifold (without boundary). The proof of this uses the notion of a generalized manifold: see Chewning and Owen \cite{chewningowen1} and the references therein.
\end{proof}

Given an isolated invariant set $K$ we may find an isolating neighbourhood $N$ with closed $N^i$ and $N^o$ by Churchill's result. Then by the proposition above there exists a compact $2$--manifold $P$ in $O$ which is a neighbourhood of $n^+$. Define again $B_0 := \{p \in N : \pi^i(p) \in P\} \cup N^-$ and make it convex to the flow by using a small collar of $\partial P$ in $O$; then throw away the remaining part of $N$. The result is an isolating block $B$ for $K$ contained in $N$ (of course, not necessarily regular). \label{pg:constructiblock} For later reference we observe that $B$ can be made as big as needed: if $D$ is any compact subset of the interior of $N$ then one can achieve that $B$ contain $D$ in its interior. To check this first observe that $\pi^i(D) \cup n^+$ is a compact subset of $N^i$ and, since $D \subseteq {\rm int}\ N$, each point in $\pi^i(D)$ is a transverse entry point by the same argument given above for points in $n^+$. Thus we may choose $P$ to contain $\pi^i(D) \cup n^+$ in its interior, and then $B_0$ contains $D$ in its interior. The claim follows.

Thus any isolated invariant set $K$ has a neighbourhood basis of isolating blocks, and as a consequence of Theorem \ref{teo:asas} any such $K$ with a finitely generated cohomology has a neighbourhood basis of regular isolating blocks. \label{pg:ribs}

\section{Identifying regular isolating blocks \label{sec:concentric}}

The definition of a regular isolating block only stipulates what its cohomology should be, and in general that does not determine the topological type of a $3$--manifold except in very simple cases. For example if $K$ 
 has the cohomology of a point then so does a regular isolating block $B$, and then Lefschetz duality shows that $\partial B$ is a $2$--sphere. If the phase space is $\mathbb{R}^3$ it then follows from the tameness of $\partial B$ and the Sch\"onflies theorem that $B$ is a ball, as one would expect.

In this section we use different (non-homological) methods to prove that a regular isolating block is completely determined by $K$ up to ambient isotopy and even independently of the flow. As a consequence of this, in many cases it is possible to identify how a regular isolating block looks like without knowing what the dynamics are. For the whole section we assume that $K$ has regular isolating blocks, or equivalently that $\check{H}^*(K)$ is finitely generated (discussion at the end of the preceding section).

The following definition is not standard. Let $K$ be a compact subset of a $3$--manifold and let $N$ be a compact neighbourhood of $K$ that is a tame $3$--manifold. We say that $K$ is a spine of $N$ if there exists a homeomorphism $N - K \cong (\partial N) \times [0,+\infty)$ that sends every $p \in \partial N$ to $(p,0)$. We begin with the following result:

\begin{theorem} \label{teo:spine} If $B$ is a regular isolating block for an isolated invariant set $K$, then  $K$ is a spine for $B$.
\end{theorem}

The proof requires a preliminary construction. The left panel of Figure \ref{fig:spindle} shows a triangular region $T$ missing its left side; namely $T := \{(x,y) \in \mathbb{R}^2 : 0 < x \leq 1, |y| \leq 1-x\}$. Denote by $\partial T$ the boundary of $T$ as a manifold; i.e. the two thick sides which are half open intervals. There is an obvious homeomorphism $T \longrightarrow (\partial T) \times [0,1)$ whose first component is given by radial projection from $(0,0)$ and whose second component is essentially the relative position of a point in its projecting ray. 

Consider the right panel of Figure \ref{fig:spindle}. It shows the same triangle $T$ together with an auxiliary dotted curve. We use this curve to fiber $T$ with disjoint rays emanating from the origin as follows. We start from the origin in any direction and draw a straight line. If it hits the dotted curve, we bend the ray vertically upwards (or downwards, in the bottom half of the figure) until it hits $\partial T$. If the ray never hits the dotted curve we just prolong it until it hits $\partial T$. This produces a partition of $T$ into rays and we use it to define a homeomorphism $f = (f_1,f_2) : T \longrightarrow (\partial T) \times [0,1)$ as follows. Given $p \in T$, find the ray passing through it and follow it up to $\partial T$; the endpoint is $f_1(p)$. To obtain $f_2(p)$ just take the distance of $p$ to $\partial T$ as measured along the ray and divide it by the total length of the ray. Thus $f_1$ is some sort of ``radial projection'' from the origin but with bent rays. The crucial property of these rays is how the behave near a point $(0,y)$, with $y \neq 0$, on the vertical axis: there is a small disk around the point such that all rays crossing that disk are already vertical. This implies that if $(x_k,y_k) \rightarrow (0,y)$ with $y > 0$ (say) then $f_1(x_k,y_k) = (x_k,1-x_k)$ for $k$ large enough. Also, since the projecting ray through $(x_k,y_k)$ is almost a straight vertical line of length $1$, $f_2(x_k,y_k)$ is approximately $1-y_k$ and in fact converges to $1-y$ as $k \rightarrow +\infty$.

Rotating this configuration in $3$--space around the vertical axis produces a sort of spindle $S$ missing its symmetry axis. The open interval $(0,1]$ of the horizontal axis produces, under this rotation, a once-punctured (at the origin) closed disk $F$. The map $f$ extends in the obvious way to a homeomorphism of $S$ onto $(\partial S) \times [0,1)$.

\begin{figure}[h]
\null\hfill
\subfigure[]{
	\begin{pspicture}(0,0)(4.5,8)
	\rput[lb](0,0){\scalebox{0.8}{\includegraphics{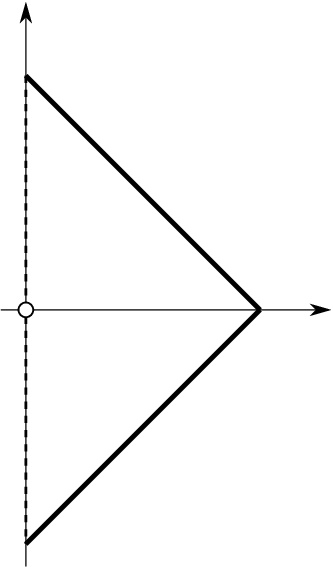}}}
	\end{pspicture}}
\hfill
\subfigure[]{
	\begin{pspicture}(0,0)(4.5,8)
	\rput[lb](0,0){\scalebox{0.8}{\includegraphics{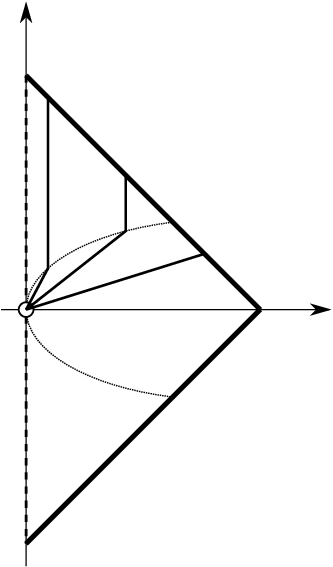}}}
	\end{pspicture}}
\hfill\null
\caption{ \label{fig:spindle}}
\end{figure}

\begin{proof}[Proof of Theorem \ref{teo:spine}] Suppose $B$ is a regular isolating block for $K$. We want to define a homeomorphism $h : B - K \longrightarrow \partial B \times [0,1)$. Let $u(p)$ be the height function of $B$ and recall that it is well defined over all of $B$ except for its tangency curves; i.e. $\partial B^i$. In particular points in $B^+ \cup B^-$ have a well defined height and we begin by defining $h(p)$ for points there by

\[h(p) := \left\{ \begin{array}{cc} (\pi^i(p),1-u(p)) & \text{ if } p \in B^+ - K \\ (\pi^o(p), 1+u(p)) & \text{ if } p \in B^- - K \end{array} \right.\]

This is a homeomorphism between $(B^+ \cup B^-)- K$ and $(b^+ \cup b^-) \times [0,1)$. To extend it to the rest of $B$ we do as follows. Concentrate on any one component $A$ of $B^i - b^+$.
\smallskip

{\it Claim.} $A$ is a once-punctured closed disk. \label{pg:punctured}

{\it Proof of claim.} Observe that $B^i - b^+$ is open in $B^i$ and so it is a surface with boundary $\partial (B^i - b^+) = \partial B^i$. Also, by Alexander duality one has $H_k(B^i-b^+,\partial B^i) = \check{H}^{2-k}(B^i,b^+) = 0$ where the latter equality owes to the assumption that $B$ is a regular isolating block for $K$. This shows that each component $A$ of $B^i - b^+$ contains exactly one component $\gamma$ of $\partial B^i$ and, moreover, $H_k(A,\gamma) = 0$. Thus capping $\gamma$ with a disk produces a connected, boundariless surfacewhich is noncompact and has trivial homology. It follows from the classification of noncompact surfaces (\cite[Theorem 3, p. 268]{richards1}) that it is an open disk, and so $A$ is once-punctured closed disk. $_{\blacksquare}$
\medskip

Let $U$ be the set of points $p \in B$ whose trajectory enters $B$ through $A$; i.e. $U := \{p \in B : \pi^i(p) \in A\}$. Recall that the spindle $S$ has as its equatorial section a once-punctured closed disk $F$. Pick a homeomorphism $a : A \longrightarrow F$ and extend it to all of $U$ setting $a(p) := a(\pi^i(p))$. We use this to construct a homeomorphism $g$ of $U$ onto the spindle $S$ just by sending the trajectory segment of a point $p \in U$ onto the vertical line through $a(p)$ in the spindle: $g(p) := (a(p),(1-\|a(p)\|)u(p))$. For points $p \in \partial A$ the height $u(p)$ is not well defined and we set $g(p) := (a(p),0)$. This map is obviously continuous at points $p \not\in \partial A$; it is also continuous at $p \in \partial A$ because $|(1-\|a(q)\|) u(q)| \leq 2(1 - a(q))$ which goes to zero as $q \rightarrow p$. Notice that as $p \not\in \partial A$ ranges over a trajectory segment in $U$ the number $u(p)$ goes from $1$ to $-1$ but the point $a(p)$ remains constant, and so $g(p)$ indeed runs vertically from the top of the spindle to the bottom. If $p \in \partial A$ then its trajectory segment in $U$ is just $p$ itself, which is again the (degenerate) vertical line through $a(p)$ in the spindle $S$. 

Now we copy the ``bent-ray'' radial projection $f : S \longrightarrow (\partial S) \times [0,1)$ from the spindle to $U$ defining $h = (h_1,h_2) : U \longrightarrow (\partial U) \times [0,1)$ by $h_1 = g^{-1} \circ f_1 \circ g$ and $h_2 := f_2 \circ g$. This is obviously a homeomorphism that takes each point $p \in \partial U$ to $(p,0)$. Observe the following: if $p_k$ is a sequence of points in $U \subseteq N$ converging to a point $p \in B^+ - K$ we have that $u(p_k) \rightarrow u(p) > 0$ and $a(p_k) \rightarrow 0$ (this is forced for topological reasons). Thus $g(p_k) \rightarrow (0,u(p))$ and by the properties of the projection $f$ mentioned above $f_1(g(p_k)) = (a(p_k),1-\|a(p_k)\|)$ and $f_2(g(p_k)) \rightarrow 1 - u(p)$. Upon applying $g^{-1}$ applying $g^{-1}$ onto $f_1(g(p_k))$ we recover the point in the trajectory segment of $p_k$ at height $u = 1$, i.e. $\pi^i(p_k)$. Thus $h(p_k) \rightarrow (\pi^i(p),1-u(p))$.   
\end{proof}

Suppose that two compact $3$--manifolds $N'$ and $N$ are nested; that is, $N'$ is contained in the interior of $N$. Then $N$ and $N'$ are called concentric if there exists a homeomorphism $N - {\rm Int}\ N' \cong (\partial N) \times [0,1]$. We will make use of the following result of Edwards (\cite[Theorem 2, p. 419]{edwards2}): let $M_0 \subseteq M \subseteq M_1$ be compact $3$--manifolds with homeomorphic nonempty boundaries, with $M_0$ and $M$ tamely imbedded in ${\rm Int}\ M$ and ${\rm Int}\ M_1$ respectively. If $M_0$ and $M_1$ are concentric, then $M$ is concentric with both $M_0$ and $M_1$.

\begin{lemma} \label{lem:concentric} Let $N'$ be contained in the interior of $N$ and suppose that $K$ is a spine for both $N$ and $N'$. Then $N$ and $N'$ are concentric. In particular, they are ambient isotopic relative to $K$.
\end{lemma}
\begin{proof} Let $h : N - K \longrightarrow (\partial N) \times [0,+\infty)$ and $h' : N' - K \longrightarrow (\partial N') \times [0,+\infty)$ be the homeomorphisms provided by the condition that $K$ be a spine for both $N$ and $N'$. For every positive real number $r$ denote by $N_r := h^{-1}((\partial N) \times [r,+\infty)) \cup K$. These form a decreasing (with increasing $r$) neighbourhood basis of $K$. Define $N'_r$ in a similar fashion.
\smallskip

{\it Claim.} The inclusion $N' - K \subseteq N - K$ induces isomorphisms in cohomology with $\mathbb{Z}$ coefficients.

{\it Proof of claim.} Pick $r$ and $s$ so big that $N_s \subseteq N'$ and $N'_r \subseteq N_s$. Then \[N'_r - K \subseteq N_s - K \subseteq N' - K \subseteq N - K\] and this chain of inclusions induces a chain of homomorphisms in $q$--dimensional cohomology \begin{equation} \label{eq:inclusions} H^q(N'_r - K) \stackrel{\bf 1}{\longleftarrow} H^q(N_s - K) \stackrel{\bf 2}{\longleftarrow} H^q(N' - K) \stackrel{\bf 3}{\longleftarrow} H^q(N - K)\end{equation}

Consider the commutative diagram \[\xymatrix{N'_r - K \ar[d]^{\cong}_{h'|} \ar[r] & N'-K \ar[d]^{\cong}_{h'} \\ \partial N' \times [r,+\infty) \ar[r] & \partial N' \times [0,+\infty)}\] where the horizontal unlabeled arrows are inclusions and the left vertical arrow is the appropriate restriction of $h'$. The vertical arrows are homeomorphisms. Clearly the lower inclusion induces an isomorphism in $H^q$, and it follows that so does the upper arrow. But the homomorphism induced by the upper arrow is precisely the composition of {\bf 2} and {\bf 1} in Equation \eqref{eq:inclusions} above, so it follows that {\bf 2} is injective. An entirely analogous argument (now using $N - K$ and $N_s - K$) shows that the composition of {\bf 3} and {\bf 2} is also an isomorphism and therefore {\bf 2} is surjective. Thus {\bf 2} is an isomorphism, and then using again that the composition of {\bf 3} and {\bf 2} is an isomorphism, it follows that {\bf 3} itself must be an isomorphism too, as was to be shown. $_{\blacksquare}$
\smallskip

Let $M_0 = N_r$, $M = N'$ and $M_1 = N$. Clearly $N_r$ and $N$ are concentric, because $h$ provides a homeomorphism between $N - {\rm Int}\ N_r$ onto $(\partial N) \times [0,r]$. Also, it follows from the previous claim and the homeomorphisms provided by $h$ and $h'$ that the closed surfaces $\partial N$ and $\partial N'$ have isomorphic cohomology rings and hence are homeomorphic. Using the cohomology ring structure accounts for the possible non connectedness of $\partial N$ and $\partial N'$: cup product with the obvious generators of $H^0$ remembers how the elements in $H^1$ are distributed among the components of the surface. It then follows from the theorem of Edwards quoted above that $N'$ and $N$ are concentric.

Showing that two concentric manifolds are ambient isotopic is a routine argument which we just outline. Set $C_0 := \overline{N - N'}$ and let $C_1$ be a collar of $N$. (Recall that our manifolds are assumed to be tame, which is why $C_1$ exists). Now since $N$ and $N'$ are concentric, $C_0 \cup C_1$ is a collar of $N'$. Thus there exists an ambient isotopy which takes $N'$ to $N' \cup (C_0 \cup C_1)$. The latter is equal to $(N' \cup C_0) \cup C_1 = N \cup C_1$, and then there is an ambient isotopy which takes this to $N$. The concatenation of these isotopies produces the required isotopy sending $N'$ to $N$.
\end{proof}

\begin{theorem} Any two regular isolating blocks for $K$ (even for different flows) are ambient isotopic relative to $K$.
\end{theorem}
\begin{proof} Let $B$ and $B'$ be regular isolating blocks for $K$ under the flows $\varphi$ and $\varphi'$. We may find another regular isolating block $B''$ for the flow $\varphi$ and contained in the interiors of both $B$ and $B'$. By the spine theorem $K$ is a spine of these three isolating blocks. Then applying the lemma above we have that $B''$ and $B$ are ambient isotopic, and so are $B''$ and $B'$. Thus (by concatenating isotopies) we have that $B$ and $B'$ are ambient isotopic.
\end{proof}

Of course, the colouring (that is, the entry and exit sets) of a regular isolating block \emph{does} depend on the flow. It is only the underlying manifold and its ambient isotopy class which do not. (One can show that if $B_1$ and $B_2$ are two regular isolating blocks for $K$ for the same flow then there exists an ambient isotopy that carries $B_1$ onto $B_2$ and is colour preserving; i.e. it carries $B_1^i$ and $B_1^o$ onto $B_2^i$ and $B_2^o$. Thus in the handle theorem one can say that, up to a colour preserving ambient isotopy, $N$ can be obtained from any regular isolating block for $K$).

\begin{corollary} \label{cor:regular_tame} Let $K$ be a tame set. Then any regular isolating block for $K$ is ambient isotopic to a regular neighbourhood for $K$ in the sense of piecewise linear topology.
\end{corollary}
\begin{proof} After performing an ambient homeomorphism we may take $K$ to be a polyhedron. Let $N$ be a regular neighbourhood for $K$ in the sense of piecewise linear topology. It is then very easy to construct a flow $\varphi$ that has $K$ as a stable attractor with $N$ being a positively invariant neighbourhood of $K$ and such that the flow crosses the boundary of $N$ transversally (see \cite[Proposition 12, p. 6169]{mio2} for example). Since $K \subseteq N$ is a homotopy equivalence, $N$ is in particular a regular isolating block for $K$ and by the previous theorem any regular isolating block for $K$ is ambient isotopic to $N$.
\end{proof}

\begin{example} \label{ex:reg_dyn_pl} A regular isolating block for a (tame) graph of genus $g$ is a handlebody of genus $g$. In particular, a regular isolating block for a fixed point is a ball and a regular isolating block for a tame invariant knot is a solid torus having the knot as its core.
\end{example}

\section{Proving the existence of invariant structure from the presence of invariant knots} \label{sec:proof}

We come back to the initial motivation of using invariant knots $K$ in some region $N$ to establish the existence of additional invariant structure in $N$. We first prove Theorem \ref{thm:detect} below which is a particular case of Theorem \ref{thm:detect1} but contains the geometric essence of the argument. Then we build on it to prove Theorem \ref{thm:detect1}. Instead of knots we consider a wider class of compacta, called knot-like.

\subsection{} \label{subsec:detect} A knot-like compactum $K \subseteq \mathbb{R}^3$ is a set having a neighbourhood basis of tame concentric solid tori $\{T_k\}$. This implies that $K$ is connected and $\check{H}^1(K) = \mathbb{Z}$. A knot-like compactum has a well defined knot type as follows. Observe that if $\{T'_{\ell}\}$ is another nested neighbourhood basis of $K$ comprised of (tame) solid tori, by interlacing them with the $T_k$ it follows from the concentricity theorem of Edwards quoted above that the $T'_{\ell}$ will be concentric with the $T_k$ (and therefore also among themselves) for large enough $\ell$. Thus we may correctly define the knot type of $K$ as the knot type of a polygonal core curve $c$ of any of the solid tori $T_k$, or of any $T'_{\ell}$ for large enough $\ell$. We call $c$ a polygonal model of $K$. Any tame knot is evidently knot-like and the notion of knot type just defined coincides with the usual one for tame knots. If $K_1, \ldots, K_r$ are disjoint knot-like compacta, we call $K = K_1 \cup \ldots \cup K_r$ a link of knot-like compacta. Its polygonal model is defined in the obvious way.

\begin{theorem} \label{thm:detect} Let $N \subseteq \mathbb{R}^3$ be an isolating block that contains an invariant link $K$ of knot-like sets. Assume that $K$ is contractible and nontrivial in $N$. Then every neighbourhood $U$ of $K$ contains a point $p \in U - K$ such that the full trajectory of $p$ is contained in $N$.
\end{theorem}

Of course, the assumptions about $K$ mean that its polygonal model is contractible and nontrivial in $N$.

To illustrate why knot-like compacta may be useful, suppose one observes the following invariant structure inside $N$: a nontrivial knot $K$, a fixed point $q$, and an orbit $\gamma$ connecting the two; i.e. its $\alpha$-limit is $q$ and its $\omega$-limit is contained in $K$. It might spiral towards $K$ or perhaps converge to some fixed point in it; that does not affect the argument. See Figure \ref{fig:knotlike}.(a) or (b). If we apply Theorem \ref{thm:detect} to $K$ we cannot detect any new invariant structure since any point $p \in \gamma$ satisfies the conclusion of the theorem for any $U$ and we already knew about $\gamma$. However, we can also apply the theorem to the knot-like invariant set $K \cup \gamma \cup q$ (which has the same knot type as $K$, hence nontrivial) and deduce that there exist full trajectories of the flow in $N$ disjoint from $K \cup \gamma \cup q$ and passing arbitrarily close to it. Now these orbits are different from $\gamma$ and so we do gain new information.

\begin{figure}[h!]
\null\hfill
\subfigure[]{
\begin{pspicture}(0,0)(3.2,4)
	\rput[bl](0,0){\includegraphics{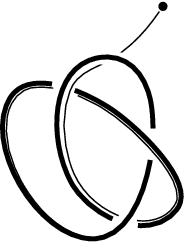}}
\end{pspicture}}
\hfill
\subfigure[]{
\begin{pspicture}(0,0)(3.8,4)
	\rput[bl](0,0){\includegraphics{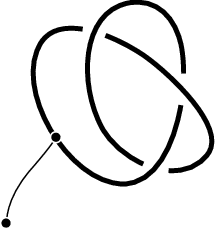}}
\end{pspicture}}
\hfill\null
\caption{ \label{fig:knotlike}}
\end{figure}

The proof of Theorem \ref{thm:detect} needs some previous lemmas.

\begin{lemma} \label{lem:identify} Let $K$ be an isolated invariant set which is a link of knot-like sets $K_i$. Then any regular isolating block $T$ for $K$ is a disjoint union of solid tori $T_i$, each along $K_i$.
\end{lemma}
\begin{proof} Since $T$ is a regular isolating block for $K$, it consists of components $T_i$ each of which is a regular isolating block for $K_i$. On the other hand each $K_i$, being knot-like, has a neighbourhood basis $\{T_{ki}\}$ of concentric solid tori. Concentricity clearly implies that $K_i$ is a spine of each $T_{ki}$; since it is also a spine of $T_i$ by Theorem \ref{teo:spine} it follows from Lemma \ref{lem:concentric} that the $T_{ki}$ and $T_i$ are concentric. In particular each $T_i$ is a solid torus and its core curve has the knot type of $K_i$.   
\end{proof}

The following lemmas are concerned with coloured manifolds, with no dynamics involved. All objects are assumed to be polyhedral.

\begin{lemma} \label{lem:tN0} Let $T$ be a coloured solid torus with a core curve $c$. There exists an integer $m \geq 1$ such that each $t$--curve of $T$ is either zero or $\pm mc$ in $H_1(T)$.
\end{lemma}
\begin{proof} If all $t$--curves are zero in $H_1(T)$ we are finished. Assume some $t$--curve $\tau$ is nonzero in $H_1(T)$, so of the form $\pm mc$ with $m \geq 1$. In particular $\tau$ is nonzero in $H_1(\partial T)$ and so cutting $\partial T$ along it produces an annulus which contains all the remaining $t$--curves. Clearly each of these will then either bound a disk or be parallel to the boundary of the annulus. Thus back in $\partial T$, every $t$--curve either bounds a disk or is parallel to $\tau$. Hence their homology class in $H_1(T)$ is either zero or the same as $\tau$ (up to a sign); i.e. of the form $\pm mc$.
\end{proof}

\begin{lemma} \label{lem:curve} Let $T$ be a coloured solid torus with a core curve $c$. Let $B$ be another coloured manifold (disjoint from $T$) and suppose $N$ is obtained from $T \cup B$ by pasting coloured handles onto it. Assume $c = 0$ in $H_1(N)$. Then there exists a $t$--curve in $T$ which is parallel to $c$; i.e. they cobound an annulus in $T$.
\end{lemma}
\begin{proof} Let $N^{(1)}$ be the coloured manifold obtained after attaching all the $1$--handles, but not the $2$--handles yet, onto $T \cup B$. By Remark \ref{rem:thom} \[H_1(N^{(1)}) = \langle c \rangle \oplus H_1(B) \oplus \langle h_1,\ldots,h_r \rangle. \]

By Lemma \ref{lem:tN0} every $t$--curve in $T$ is either zero or $\pm mc$ in $H_1(T)$, and so the formulae in Remark \ref{rem:thom} show inductively that every $t$--curve in $N^{(1)}$ has an expression of the form $kmc + x + y$ where $x \in H_1(B), y \in \langle h_1,\ldots,h_r \rangle$, and $k \in \mathbb{Z}$. Notice also that if all $t$--curves in $T$ are zero in $H_1(T)$ then $k = 0$.

The homology of $N$ is the quotient of $H_1(N^{(1)})$ by the subgroup generated by the $t$--curves onto which the $2$--handles are attached. All of these have the above form $k_i mc + x_i + y_i$. Since $c = 0$ in $H_1(N)$ we must have (the homology class of) $c$ belong to this subgroup, and this forces that at least one $k_i \neq 0$ and also that $m = 1$. In particular, at least one of the $t$--curves of $T$ has the form $\pm c$ and so cobounds an annulus with $c$.
\end{proof}

\begin{lemma} \label{lem:boundary} Consider a coloured manifold $N_0$ which is a disjoint union $T_1 \cup \ldots \cup T_r \cup B$ where each $T_i$ is a solid torus with a core curve $c_i$. Let $N$ be obtained from $N_0$ by pasting handles and assume each $c_i = 0$ in $H_1(N)$. Let $\partial N = P \cup Q$ be the colouring of $N$. Then there exists a link $\cup \tau_i$ entirely contained in $P$ (similarly, entirely contained in $Q$) which is parallel in $N$ to $\cup c_i$; i.e. there exist disjoint annuli $A_i \subseteq N$ such that each pair of curves $c_i$ and $\tau_i$ cobound $A_i$.
\end{lemma}
\begin{proof} Applying the preceding lemma to the expression $T_i \cup (T_1 \cup \ldots \cup T_{i-1} \cup T_{i+1} \cup \ldots \cup B)$ we see that each $T_i$ contains a $t$--curve which cobounds an annulus with $c_i$ in $T_i$. Push these $t$--curves slightly into the (say) white subset of $\partial T_i$ to obtain the curves $\tau_i$. Clearly the link $\cup c_i$ is parallel in $N_0$ to $\cup \tau_i$. We now paste the handles onto $N_0$ to obtain $N$. Since the handles are all pasted onto disks or annuli centered on $t$--curves, by making these sufficiently small (for the disks) or thin (for the annulus) we may assume without loss of generality that they are all disjoint from the $\tau_i$. Thus the link $\cup \tau_i$ is contained in $P$ (the white subset of $\partial N$) and it is obviously still parallel to $\cup c_i$ since we have only enlarged the manifold.
\end{proof}

We recall a standard result from $3$--manifold theory.  Let $N$ be a compact manifold and $\tau$ a simple closed curve in $\partial N$. If $\tau$ is contractible in $N$, then it bounds an embedded disk $D$ that is properly embedded in $N$; that is, $D \cap \partial N = \partial D$. This is a consequence of Dehn's lemma (\cite[Corollary 2, p. 101]{rolfsen1}). As a generalization of this, suppose $\tau_1, \ldots, \tau_r$ are disjoint simple closed curves in $\partial N$ all of which are contractible in $N$. As we have just seen each $\tau_i$ bounds a properly embedded disk $D_i \subseteq N$, and these can be taken to be mutually disjoint. This can be achieved by a standard procedure so we only describe it very briefly. Consider $D_1$ and $D_2$; after a small perturbation we can assume them to intersect transversally. Thus $D_1 \cap D_2$ is a finite collection of disjoint simple closed curves. Consider these intersection curves as subsets of $D_1$ and pick an innermost one $\alpha$; by this we mean that the closed disk $E_1$ bounded by $\alpha$ in $D_1$ contains no other intersection curve in its interior. Now regard $\alpha$ as a subset of $D_2$, where it bounds a closed disk $E_2$ (which may contain other intersection curves in its interior) and replace $D_2$ with $(D_2 - {\rm int}\ E_2) \cup E_1$. This is a new disk whose boundary is still $\tau_2$ and now intersects $D_1$ along all the remaining intersection curves other than $\alpha$ and also along the whole disk $E_1$. Now a slight push of this disk in a direction normal to $D_1$ (towards the side opposite $E_2$) will also remove $E_1$ from its intersection with $D_1$. Thus $D_1$ and the new $D_2$ now intersect (transversally) along the same curves that they did before minus $\alpha$. Repeating this same process will eventually produce a disk $D_2$ disjoint from $D_1$ and still having $\tau_2$ as its boundary while leaving $D_1$ unchanged. If we have a third disk $D_3$ we continue by first making $D_1$ and $D_3$ disjoint by the same procedure just explained. Then we make $D_2$ and $D_3$ disjoint again by the same procedure and observe that this only involves replacing pieces of $D_3$ with pieces of $D_2$ and, since $D_1$ is disjoint from $D_2$ and $D_3$, these replacements produce a new disk $D_3$ that is still disjoint from $D_1$ (and leave $D_2$ unchanged). This extends in the obvious way to any number of disks.

\begin{proof}[Proof of Theorem \ref{thm:detect}] We reason by contradiction. Denote by ${\rm Inv}\ N$ the maximal invariant subset of $N$. If the thesis does not hold, there is a neighbourhood $U$ of $K$ such that the trajectory of every $p \in U-K$ leaves $N$. In particular, $U \cap {\rm Inv}\ N = K$ and so ${\rm Inv}\ N$ decomposes as a disjoint union $K \uplus K'$ where $K'$ is some other (possibly empty, in principle) isolated invariant set. Applying the generalized handle theorem (Theorem \ref{teo:asasgen}) there exist a regular isolating block $T$ for $K$ and an isolating block $B$ for $K'$ such that $N$ is obtained from $T \cup B$ by pasting handles. (If $K'$ is empty we just apply the ordinary handle theorem).

By Lemma \ref{lem:identify} the block $T$ is a disjoint union of solid tori $T_i$ and their cores $c_i$ form a link which is a polygonal model of $K$. By assumption all $c_i$ are contractible in $N$. In particular they are nullhomologous, so by Lemma \ref{lem:boundary} there exists a link $\cup \tau_i$ in (the white subset of) $\partial N$ such that each pair $c_i, \tau_i$ cobounds an annulus $A_i$ and these annuli are all disjoint. Each $\tau_i$ is contractible in $N$ (being freely homotopic to $c_i$ across the annulus $A_i$) so by Dehn's lemma the $\tau_i$ bound embedded disjoint disks $D_i$ in $N$.

Use a small collar of $\partial N$ to shrink $N$ ever so slightly so that the $c_i$ remain unchanged but the $\tau_i$, $A_i$, and $D_i$ are now all contained in the interior of $N$. There is a homeomorphism $h$ of $N$ such that $h(c_i) = \tau_i$ for each $i$; it can be obtained as the final stage of an isotopy that slides $c_i$ across the annulus $A_i$ until it matches $\tau_i$. Now the $c_i$ bound the embedded disks $h^{-1}(D_i)$ which are all disjoint, and so $\cup c_i$ is the trivial link. This contradicts the assumption about $K$.
\end{proof}

It is worth observing that the fact that our manifolds and handles are coloured plays an essential role in the argument. In fact, the non-coloured version of Lemma \ref{lem:curve} is actually false. As a simple example take $T$ to be the coloured torus of Figure \ref{fig:counter} (and let $B$ be empty) and let $c$ be its core. There is an obvious way to paste a non-coloured $2$--handle onto $T$ to obtain a ball $N$, and $c = 0$ in $H_1(N)$. However, the only $t$--curve in $T$ is not parallel to its core. Thus Lemma \ref{lem:curve} fails under non-coloured handle pastings. On the other hand there is only one way to paste a coloured $2$--handle, and it produces a manifold where $c$ is not homologous to zero.

\begin{figure}[h!]
\begin{pspicture}(0,0)(6.5,2.5)
	\rput[bl](0,0){\scalebox{0.6}{\includegraphics{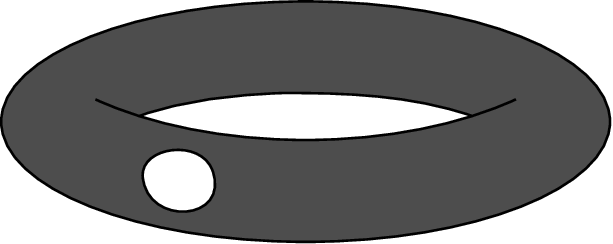}}}	
\end{pspicture}
    \caption{\label{fig:counter}}
\end{figure}

\subsection{} We now head towards a proof of Theorem \ref{thm:detect1} from the Introduction, which we restate here in a slightly more general form (for links of knot-like sets):

\begin{theorem} \label{thm:main} Let $(Q,Q_0)$ be an index pair in $\mathbb{R}^3$. Assume that $\overline{Q-Q_0}$ contains an invariant link $K$ of knot-like sets which is contractbile and nontrivial in its interior. Then every neighbourhood $U$ of $K$ contains a point $p \in U - K$ such that the full trajectory of $p$ is contained in $Q-Q_0$.
\end{theorem}

We recall the definition of an index pair. It is a compact pair $(Q,Q_0)$ such that:
\begin{itemize}
    \item[(i)] $\overline{Q-Q_0}$ is an isolating neighbourhood.
    \item[(ii)] $Q_0$ is positively invariant in $Q$; i.e. if $p \in Q_0$ and $p \cdot [0,t] \subseteq Q$ then $p \cdot [0,t] \subseteq Q_0$.
    \item[(iii)] $Q_0$ is an exit set for $Q$; i.e. if $p \in Q$ satisfies $p \cdot [0,+\infty) \not\subseteq Q$ there exists $t \geq 0$ such that $p \cdot [0,t] \subseteq Q$ and $p \cdot t \in Q_0$.
\end{itemize}

For example, if $N$ is an isolating neighbourhood with a closed immediate exit set $N^o$ then $(N,N^o) = (Q,Q_0)$ is an index pair. In general, the exit set $Q_0$ of an index pair contains the immediate exit set of $Q$ but can be much bigger. Any isolating neighbourhood $N$ contains an index pair $(Q,Q_0)$ such that $\overline{Q-Q_0}$ contains the maximal invariant subset of $N$ in its interior (see \cite[4.1.D., p. 47]{conley1} or \cite[Theorem 4.3, p. 15]{salamon1}). Moreover, there exist algorithms (\cite[\S 2.5, pp. 459ff.]{mischaikow} and references therein) that, at least for smooth flows, will decide whether a given $N$ is an isolating neighbourhood and, if it is, find an index pair $(Q,Q_0)$. In this sense the criterion given in Theorem \ref{thm:main} is ``computable''.

The proof of Theorem \ref{thm:main} consists in reducing it to Theorem \ref{thm:detect} by finding an isolating block inside $Q - Q_0$ large enough to contain $K$ and the (image of the) homotopy that contracts it to a point. This is the content of Proposition \ref{prop:reduce} below, which needs some preliminary work.

\begin{proposition} \label{prop:i_pair} Let $C$ be an isolating neighbourhood such that:
\begin{itemize}
    \item[(i)] $C^o$ is closed (thus $(C,C^o)$ is an index pair).
    \item[(ii)] $\overline{C^i} \cap C^o = \emptyset$.
\end{itemize}

Then there exists a subset $E \subseteq {\rm fr}\ C$ such that $(C,E)$ is an index pair for the reverse flow. Moreover, $E \cap C^o = \emptyset$.
\end{proposition}
\begin{proof} We first show that $t^i$ is bounded over $\overline{C^i}$. If this not were the case there would exist a sequence $p_n \subseteq \overline{C^i}$ converging to $p \in \overline{C^i}$ with $t^i(p_n) \rightarrow - \infty$. The latter implies $p \in n^-$, and by Proposition \ref{prop:structure} we have $p \cdot (-\infty,t^o(p)) \subseteq {\rm int}\ C$. In other words, $p \cdot t^o(p)$ is the first time the trajectory of $p$ hits ${\rm fr}\ C$ so in fact $t^o(p) = 0$ and $p \in C^o$. This contradicts the assumption that $\overline{C^i} \cap C^o = \emptyset$.

Define  \[E := \bigcup_{p \in \overline{C^i}} p \cdot [t^i(p),0]\] This is a union of (backwards) trajectory segments in $C$, so it is negatively invariant in $C$. It is straightforward to check that it is closed using the boundedness of $t^i$. Finally, it is contained in ${\rm fr}\ C$. Indeed, if $p \in N^i$ then $p \cdot [t^i(p),0] = \{p\} \subseteq {\rm fr}\ C$ trivially. If $p \in \overline{C^i} - C^i$, we have $t^i(p) < 0$, $t^o(p) > 0$ (because $\overline{C^i} \cap C^o = \emptyset$ by assumption) and $p \in {\rm fr}\ C$. Since $t^i(p) < 0$ there exists $\epsilon > 0$ such that $p \cdot (-\epsilon,0] \subseteq C$; since $p \not\in C^o$ and $p \in {\rm fr}\ C$ by Proposition \ref{prop:structure} we have $p \cdot (t^i(p),0] \subseteq {\rm fr}\ C$.

All of the above implies that $(C,E)$ is an index pair for the reverse flow; notice of course that it isolates the same set as $(C,C^o)$. Finally, since $E$ is a union of negative semitrajectory segments entirely contained in $C$, if it intersects $C^o$ it must do so at a final endpoint of one of those segments; i.e. there exists $p \in \overline{C^i}$ such that $p \cdot 0 = p \in C^o$. However, this contradicts the assumption.
\end{proof}

Let $(Q,Q_0)$ be an index pair. By \cite[Lemma 5.3, p. 20]{salamon1} there exists a (continuous) Lyapunov function $g : Q \longrightarrow [0,1]$ with the following properties: (i) $g(p) = 1$ if and only if $p \in Q^+$ and $\omega(p) \cap Q_0 = \emptyset$; (ii) $g(p) = 0$ exactly on $Q_0$, (iii) if $0 < g(p) < 1$ and $p \cdot [0,t] \subseteq Q$ then $g(p \cdot t) < g(p)$. Such a map $g$ is called a Lyapunov function.

For any $0 < \alpha < 1$ let $N_{\alpha} := \{p \in Q : g(p) \geq \alpha\} \subseteq Q - Q_0$. This is an isolating neighbourhood, and we show that its immediate exit set is $N_{\alpha}^o = \{p \in Q : g(p) = \alpha\}$. Suppose $g(p) = \alpha$. Since $p \not\in Q_0$ there exists $\epsilon > 0$ such that $p \cdot [0,\epsilon] \subseteq Q-Q_0$ and so $g(p \cdot t) < g(p)$ for every $t \in (0,\epsilon)$ by (iii). Thus $p \cdot (0,\epsilon) \cap N_{\alpha} = \emptyset$ so $p \in N_{\alpha}^o$. Similarly, any other point $p \in N_{\alpha}$ with $g(p) > \alpha$ will remain for some time in $Q$ and by continuity of $g$ it will remain in $N_{\alpha}$ so in particular does not belong to $N_{\alpha}^o$.

Clearly $N_{\alpha}^o$ is closed, so $t^o$ and $\pi^o$ are continuous maps. Moreover, every point in $N_{\alpha}^o \cap {\rm int}\ Q$ is a transverse exit point. Indeed, for any such point there exists $\epsilon > 0$ such that $p \cdot (-\epsilon,\epsilon) \subseteq {\rm int}\ Q$ and in particular by (iii) $p \cdot (0,\epsilon) \cap N_{\alpha} = \emptyset$ whereas $p \cdot (-\epsilon,0) \subseteq {\rm int}\ N_{\alpha}$. As a consequence of this $N_{\alpha}^i \cap (N_{\alpha}^o \cap {\rm int}\ Q) = \emptyset$ and therefore also $\overline{N_{\alpha}^i} \cap (N_{\alpha}^o \cap {\rm int}\ Q) = \emptyset$.

We also observe the following. Suppose $p \in {\rm int}\ N_{\alpha}$ eventually exits $N_{\alpha}$. It must do so through $N_{\alpha}^o$, and we claim that it does so through $N_{\alpha}^o \cap {\rm int}\ Q$. To prove this consider $N_{\alpha/2}$, which is slightly bigger than $N_{\alpha}$. The point $p$ must eventually exit it as well at some time $t$, and since $N_{\alpha/2}^o$ is closed we may apply Proposition \ref{prop:structure} and we have $p \cdot (0,t) \subseteq {\rm int}\ N_{\alpha/2}$, so in particular $p \cdot (0,t)  \subseteq {\rm int}\ Q$. Since $p$ must exit $N_{\alpha}$ before it exits the larger set $N_{\alpha/2}$, it does so through  a point in $N_{\alpha}^o \cap {\rm int}\ Q$.

\begin{proposition} \label{prop:reduce} Let $(Q,Q_0)$ be an index pair and $D$ a compact subset in the interior of $\overline{Q - Q_0}$. There exists an isolating neighbourhood $C \subseteq {\rm int}\  \overline{Q - Q_0}$ with closed $C^i$ and $C^o$ that contains $D$ in its interior.
\end{proposition}
\begin{proof} Construct $N_{\alpha}$ as described above choosing $\alpha$ so small that $D$ is contained in its interior. It is straightforward to show that $\pi^o(D) \cup n_{\alpha}^+$ is compact, and by the last observation above it is contained in $N_{\alpha}^o \cap {\rm int}\ Q$. Pick $P$ a compact neighbourhood of $\pi^i(D) \cup n_{\alpha}^+$ contained in $N_{\alpha}^o \cap {\rm int}\ Q$. Construct the new set $C := \{p \in N_{\alpha} : \pi^o(p) \in P\} \cup N_{\alpha}^+$ (this is analogous to the construction of $B_0$ in p. \pageref{pg:constructiblock}). Then $C$ is an isolating neighbourhood and, since it is a union of trajectory segments of $N_{\alpha}$, its entry and exit sets are just the intersection of those of $N_{\alpha}$ with $C$. In particular $C^o = N_{\alpha}^o \cap C = P$. Also, $\overline{C^i} \cap C^o \subseteq \overline{N_{\alpha}^i} \cap P \subseteq \overline{N_{\alpha}^i} \cap (N_{\alpha}^o \cap {\rm int}\ Q) = \emptyset$. Hence we may apply Proposition \ref{prop:i_pair} and find $E \subseteq {\rm fr}\ C$ so that $(C,E)$ is an index pair for the reverse flow. Then using again a Lyapunov function $g'$ for the reverse flow we reduce $C$ sligthly to $C_{\beta} := \{p \in C : g'(p) \geq \beta\}$ which is now an isolating neighbourhood with a closed entry set $C_{\beta}^i$. Notice as well that since $E$ is disjoint from $C^o$, by choosing $\beta$ sufficiently small we can achieve that the part of $C$ (where $g' < \beta$) we remove be disjoint from $C^o$, and so $C_{\beta}^o = C^o$ still.

Notice that by construction $C$ contains $D$ in its interior. By choosing $\beta$ sufficiently small in this last step we can achieve that $C_{\beta}$ still contains $D$ in its interior.
\end{proof}

In particular when the phase space is a $3$--manifold we may then trim $C$ down to an isolating block still containing $D$ in its interior, as explained in p. \pageref{pg:constructiblock}. We can now prove Theorem \ref{thm:detect1}:

\begin{proof}[Proof of Theorem \ref{thm:main}] By assumption there exists a polyhedral model $\cup c_i$ of $K$ that is contractible and nontrivial in the interior of $\overline{Q-Q_0}$. Let $D$ be the union of $K$ and the image of the homotopy that contracts $\cup c_i$ to a point. By the previous discussion there is an isolating block $N$ that contains $D$ in its interior and is contained in the interior of $\overline{Q - Q_0}$. By construction $K$ is contractible in $N$ (the same polyhedral model and the same homotopy work) and evidently it is a nontrivial link in $N$ since it was already nontrivial in the larger set ${\rm int}\ \overline{Q - Q_0}$. Thus Theorem \ref{thm:detect} applies to this isolating block and provides the conclusion.    
\end{proof}

\section{Concluding remarks} \label{sec:concluding}

Let us go back to the very broad problem stated at the beginning of the paper. We are given an isolating block $N$ for some flow and know about some invariant subset $K \subseteq N$; for definiteness let it be a knot. We want to find out if there is more invariant structure in $N$ or, in other words, if $K$ is the maximal invariant subset of $N$. We assume that the behaviour of the flow on $\partial N$ is known, so the colouring of $N$ is known. For the sake of brevity we write $N_0 \leadsto N$ to mean that $N_0$ and $N$ are coloured manifolds and $N$ can be obtained from $N_0$ by pasting $1$-- and $2$--coloured handles onto it.

The basic strategy that we have used in the paper is (an elaboration on) the following. Assume that $K$ is indeed the maximal invariant subset in $N$. The machinery developed in Sections \ref{sec:asas} and \ref{sec:concentric} then shows that, whatever the flow is, there is a coloured solid torus $T$ along $K$ such that $T \leadsto N$. $T$ is known up to ambient isotopy, although in principle its colouring is not known. Then if we manage to show that $T \leadsto N$ cannot be true for any colouring of $T$, we know that there is additional invariant structure in $N$.

The remark that we want to make is that this strategy is ``complete'' in the following sense: if there is a colouring of $T$ such that $T \leadsto N$, then one can construct a flow for which $N$ is an isolating block whose entry and exit sets accord the given colouring of $N$ and which has $K$ as its maximal invariant subset. The proof is not difficult but will be provided elsewhere. Thus, if the only data is the manifold $N$ coloured by the flow and the knot $K$ inside it, deciding whether $T \leadsto N$ (for some colouring of $T$) is completely equivalent to deciding whether one can ensure that there is additional invariant structure in $N$ in general (i.e. for any flow).

Of course, there are variations on this very general outline. For example, in this paper we did not assume that $K$ was the maximal invariant set in $N$ but rather that it was an isolated component of it. Observe also that the method works equally well for any compact set $K$ with a finitely generated cohomology (so that it has regular isolating blocks) and whose regular isolating blocks we can recognize.

These considerations lead to the following purely topological problem which seems interesting in itself: given two coloured manifolds $N_0 \subseteq N$, find necessary and sufficient conditions to ensure that $N_0 \leadsto N$. As part of the necessary conditions presumably one would like to associate to any coloured manifold some magnitudes that are invariant, or at least change in a controlled way, under handle additions. Here are two simple examples.

\begin{example} (1) If $N$ is a coloured manifold with colouring $\partial N = P \cup Q$, the homotopy type of the pointed spaces $(N/P,[P])$ and $(N/Q,[Q])$ is invariant under handle additions. This arises from a direct analogy with the Conley index.

(2) The following is implicit in the proof of Theorem \ref{thm:detect}. Let $(N,P,Q)$ be a coloured manifold and consider links in the interior of $N$ up to ambient (in $N$) isotopy. Denote by $\mathcal{K}(N)$ the set of link classes that can be represented by a link $\cup c_i$ which is parallel to a link $\cup \tau_i$ entirely contained in $P$ and also parallel to a link entirely contained in $Q$. Then if $N_0 \leadsto N$ one has $\mathcal{K}(N_0) \subseteq \mathcal{K}(N)$; the reason is that any link in $P$ (or $Q$) can be pushed slightly into $P$ so as to make it disjoint from $\partial P$ and then the handles can be pasted avoiding the link. Notice that $\mathcal{K}$ is not invariant under handle additions like the first example, but still changes in a controlled way.
\end{example}

Back to dynamics, and as an illustration, suppose we implement the general strategy described above using the partial invariant $\mathcal{K}$. A straightforward reinterpretation of the arguments in Subsection \ref{subsec:detect} leads to the following:

\begin{theorem*} Let $N \subseteq \mathbb{R}^3$ be an isolating block that contains an invariant link $K$ of knot-like sets. Assume that $K$ is nullhomologous in $N$. If the link type of $K$ is not represented in $\mathcal{K}(N)$ then every neighbourhood $U$ of $K$ contains a point $p \in U - K$ such that the full trajectory of $p$ is contained in $N$.
\end{theorem*}

This generalizes the main theorem in that $K$ is assumed to be nullhomologous but not necessarily contractible. (When it is, Theorem \ref{thm:main} follows from this by Dehn's lemma). Notice that by its very definition $\mathcal{K}(N)$ is computable from the colouring of $N$. For example, if $N$ is the solid torus of Figure \ref{fig:counter} then $\mathcal{K}(N)$ only contains the trivial link type (for any given number of components) because the white subset of $\partial N$ consists of a single disk which can only carry trivial links.

Any invariant (or partial invariant, as $\mathcal{K}$) under coloured handle addition would lead, with exactly the same procedure, to results structurally similar to this. Clearly, the more sensitive the invariant, the more sensitive the criterion to find additional invariant structure becomes.

\section{Appendix. Proof of Theorem \ref{teo:reg_surf_1}} \label{app:frame}

We recall the setting here. $S$ is a compact surface containing a compact set $L$ in its interior. The assumption is that $\check{H}^i(L)$ is finitely generated for $k = 0,1$. We consider three types of operations: removing a disk $E$ from $S$, cutting $S$ along a strip $E$, or removing a whole connected component $E$ from $S$. For the sake of brevity we call these operations ($R_0$), ($R_1$) and ($R_2$). $E$ is always assumed to be disjoint from $L$ so that the surface $S'$ that results after applying the operation still contains $L$ in its interior.

For geometric reasons it is best to work with coefficients in $\mathbb{Z}_2$. For $i = 0,1,2$ we let $r_i = {\rm rk}\ \check{H}^{2-i}(S,L;\mathbb{Z}_2)$, and similarly $r'_i = {\rm rk}\ \check{H}^{2-i}(S',L;\mathbb{Z}_2)$. These are finite by the long exact sequence of the pair $(S,L)$, the fact that any surface has a finitely generated (\v{C}ech) cohomology, and $L$ has a finitely generated cohomology with $\mathbb{Z}_2$ coefficients by the universal coefficients theorem. In the sequel we shall not display the $\mathbb{Z}_2$ coefficients explicitly in the notation.

Although the $r_i$ are defined in terms of \v{C}ech cohomology, the constructions which follow are best visualized when homology is used. By Alexander duality in manifolds with boundary (this may not be entirely standard; see the Appendix in \cite{mio1} for a proof) there is an isomorphism $\check{H}^i(S,L) = H_{2-i}(S-L,\partial S)$, and so $r_i = {\rm rk}\ H_i(S-L,\partial S)$.

Without loss of generality we shall assume $S$ to be triangulated and use simplicial homology. For later use we remark that every nonzero element $z \in H_1(S-L,\partial S)$ can be represented as a sum of disjoint polygonal curves $\gamma_i$ such that each $\gamma_i$ is either: (i) a simple closed curve contained in the interior of $S$ or (ii) a simple curve which meets $\partial S$ precisely at its endpoints. This follows from a straightforward manipulation with a simplicial representative of $z$ and the fact that coefficients are taken in $\mathbb{Z}_2$.
\medskip

The proof of Theorem \ref{teo:reg_surf_1} consists in using operations $R_i$ to gradually decrease the $r_i$ until they are all zero. More precisely, we shall prove the following:
\begin{itemize}
	\item[(i)] If $r_0 > 0$, a suitable operation of type $R_0$ yields an $S'$ with $r'_0 = r_0 - 1$ and $r'_i = r_i$ for $i = 1,2$.
	\item[(ii)] If $r_0 = 0$ and $r_1 > 0$, a suitable operation of type $R_1$ yields an $S'$ with $r'_0$ still zero, $r'_1 = r_1 - 1$ and $r'_2 = r_2$.
\end{itemize}

Thus applying (i) repeatedly we can trim the original $S$ down until $r_0 = 0$; then applying (ii) to this new, trimmed down $S$, we can achieve that $r_1 = 0$ while still $r_0 = 0$. It will only remain to make $r_2 = 0$, and this will just be a matter of removing superfluous components of $S$.

Figures \ref{fig:(1)eh} and \ref{fig:(2)eh} show how the process described works in a particular example. The set $L$ is the grey thin annulus. The initial $S$ of Figure \ref{fig:(1)eh}.(a) has $r_0 = 1$; removing a small disk centered at $z$ yields $S'$ with $r'_0 = 0$ as in Figure \ref{fig:(1)eh}.(b). Renaming the latter $S$ and starting afresh with it in Figure \ref{fig:(2)eh}.(a) we have $r_1 = 1$. The curve $\gamma$ generates $H_1(S - L,\partial S)$, and removing it with an operation $R_1$ the resulting $S'$ has $r'_0 = r'_1 = 0$. It does not contain any superfluous components, so $r'_2 = 0$ and consequently it is not necessary to perform any $R_2$ operations: $S'$ already satisfies $\check{H}^*(S',L) = 0$.

\begin{figure}[h]
\null\hfill
\subfigure[$z$ generates $H_0(S-L,\partial S)$]{
\begin{pspicture}(0,0)(4,3.4)
\rput[bl](0,0){\scalebox{0.5}{\includegraphics{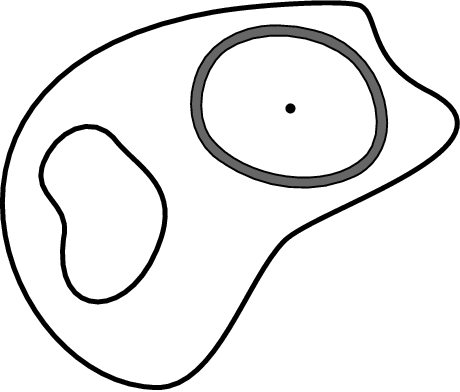}}}
\rput[bl](2,0.2){$S$} \rput[r](1.5,2.5){$L$} \rput(2.65,2.25){$z$}
\end{pspicture}}
\hfill
\subfigure[Type $R_0$ operation: remove a small disk $E$ centered at $z$]{
\begin{pspicture}(0,0)(4,3.4)
\rput[bl](0,0){\scalebox{0.5}{\includegraphics{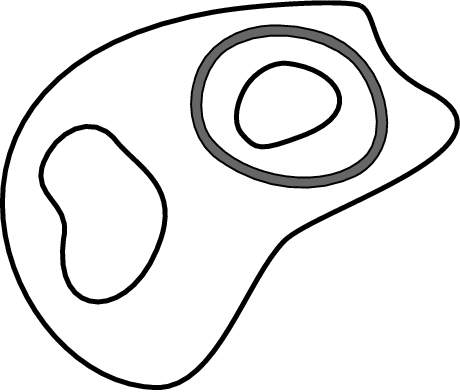}}}
\rput[bl](2,0.2){$S'$} \rput[r](1.5,2.5){$L$}
\end{pspicture}}
\hfill\null
\caption{Making $r_0 = 0$ \label{fig:(1)eh}}
\end{figure}
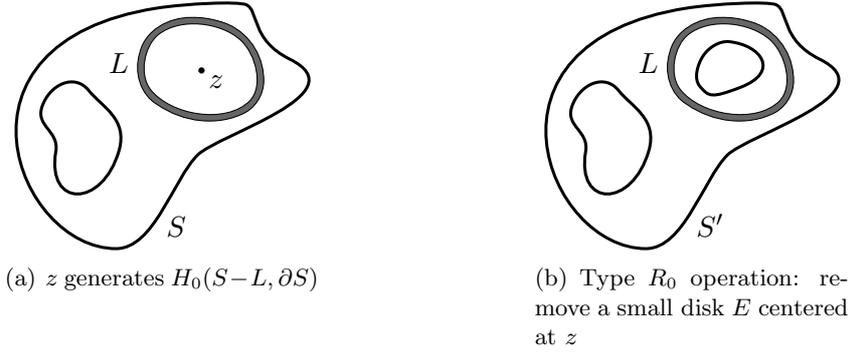

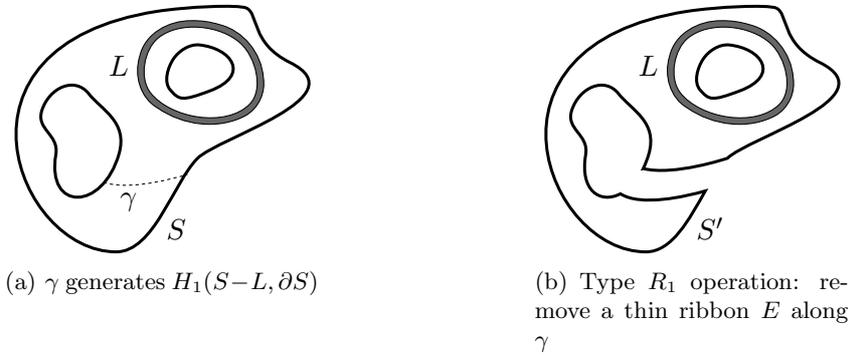
\begin{figure}[h]
\null\hfill
\subfigure[$\gamma$ generates $H_1(S-L,\partial S)$]{
\begin{pspicture}(0,0)(4,3.4)
\rput[bl](0,0){\scalebox{0.5}{\includegraphics{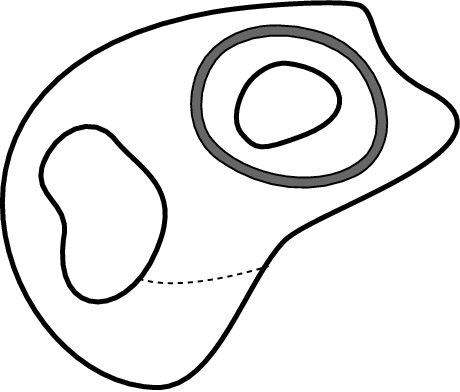}}}
\rput[bl](2,0.2){$S$} \rput[r](1.5,2.5){$L$} \rput(1.5,0.7){$\gamma$}
\end{pspicture}}
\hfill
\subfigure[Type $R_1$ operation: remove a thin ribbon $E$ along $\gamma$]{
\begin{pspicture}(0,0)(4,3.4)
\rput[bl](0,0){\scalebox{0.5}{\includegraphics{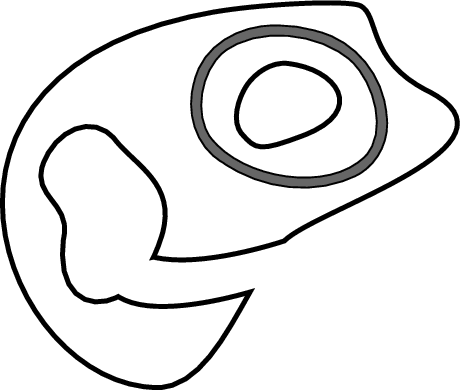}}}
\rput[bl](2,0.2){$S'$} \rput[r](1.5,2.5){$L$}
\end{pspicture}}
\hfill\null
\caption{Making $r_1 = 0$ \label{fig:(2)eh}}
\end{figure}

We begin with a simple proposition which will be used to relate $r_i$ to $r'_i$.

\begin{proposition} \label{prop:exacta} There exists an exact sequence \[\ldots \longrightarrow H_k(E,E \cap \partial S) \stackrel{j_*}{\longrightarrow} H_k(S-L,\partial S) \stackrel{f}{\longrightarrow} H_k(S'-L,\partial S') \longrightarrow \ldots\] where $j$ denotes the inclusion.
\end{proposition}
\begin{proof} Denote by $E_*$ either the center of $E$ (if $E$ is a disk) or the central arc of $E$ (if $E$ is a closed strip). Clearly $(S-L-E_*,\partial S \cup E - E_*)$ strongly deformation retracts onto $(S'-L,\partial S')$ and therefore \begin{equation} \label{eq:surf_primas} H_k(S'-L,\partial S') = H_k(S-L-E_*,\partial S \cup E - E_*) = H_k(S-L,\partial S \cup E)\end{equation} because of the excision property. Now write the pair $(S-L,\partial S \cup E)$ as the union $(S-L,\partial S) \cup (E,E)$: the Mayer--Vietoris sequence corresponding to this decomposition reads \[\ldots \longrightarrow H_k(E,E \cap \partial S) \longrightarrow H_k(S-L,\partial S) \longrightarrow H_k(S-L,\partial S \cup E) \longrightarrow \ldots\] and replacing the last group with $H_k(S'-L,\partial S')$ by \eqref{eq:surf_primas} we obtain the desired exact sequence.
\end{proof}

Let us start with the proof of Theorem \ref{teo:reg_surf_1} proper.
\medskip

{\it Step 1.} Suppose first that there exists a nonzero element in $\check{H}^2(S,L) = H_0(S - L,\partial S)$. This means that there exists a component $C$ of $S - L$ that does not meet $\partial S$. Choose any point $z \in C$ and a small closed disk $E \subseteq C$ centered at $z$. Perform an operation of type $R_0$ deleting $E$ from $S$ to obtain $S'$. From Proposition \ref{prop:exacta} we have the exact sequence \[\ldots \longrightarrow H_0(E) \stackrel{j_*}{\longrightarrow} H_0(S-L,\partial S) \stackrel{f}{\longrightarrow} H_0(S'-L,\partial S') \longrightarrow 0.\] By construction $[z]$ generates ${\rm im}\ j_*$, so $\ker f$ has rank $1$ and consequently $r'_0 = r_0 - 1$. Notice also that $\ker j_* = 0$, so $r'_1 = r_1$. It is also easy to check that $r'_2 = r_2$.

Now we repeat this process starting with $S'$ to obtain (through an $R_0$ operation on $S'$) another surface $S''$ with $r''_0 = r_0 - 2$. Notice that $S''$ can alternatively be thought of as arising from $S$ from two consecutive $R_0$ operations with disjoint supports. After $r_0$ repetitions we will be left with a surface $S^{(r_0)}$ which arises from the original one through a sequence of $r_0$ operations of type $R_0$ and satisfies that $\check{H}^2(S^{(r_0)},L) = 0$. For notational ease we now discard the original $S$ and rename $S^{(r_0)}$ as $S$. This concludes Step 1.
\medskip

{\it Step 2.} At the end of the previous step we were left with a surface $S$ that contains $L$ in its interior and such that $\check{H}^2(S,L) = 0$, which by Alexander duality means that every connected component of $S - L$ intersects $\partial S$. Suppose that there is a nonzero element in $\check{H}^1(S,L) = H_1(S - L, \partial S)$.
\smallskip

{\it Claim.} There exists a polygonal simple curve $\gamma \subseteq S - L$ which meets $\partial S$ precisely at its endpoints and such that $[\gamma]$ is nonzero in $H_1(S-L,\partial S)$.

{\it Proof of claim.} Pick a nonzero element $z \in H_1(S-L,\partial S)$. We may write $z \sim \sum_i \gamma_i$ where the $\gamma_i$ are disjoint simple curves, each $\gamma_i$ either (i) being closed or (ii) meeting $\partial S$ precisely at its endpoints. Notice that each $\gamma_i$ is an element of $H_1(S-L,\partial S)$ and at least one of them must be nonzero, say $\gamma_{i_0}$. If $\gamma_{i_0}$ is of type (ii) we let $\gamma = \gamma_{i_0}$. If it is a simple closed curve we argue as follows. Since $\check{H}^2(S,L) = H_0(S - L,\partial S) = 0$ by hypothesis, each connected component of $S - L$ meets $\partial S$. It is then easy to see that there is a simple polygonal path in $S - L$ joining a vertex $v$ of $\gamma_{i_0}$ to a point in $\partial S$ and having no more intersections with $\gamma_{i_0}$. Replacing this path by two very close parallel copies of itself we obtain, as shown in Figure \ref{fig:closed_not}.(b), a simple path $\gamma$ homologous to $\gamma_{i_0}$ and having intersection with $\partial S$ precisely at its endpoints. This concludes the proof of the claim. $_{\blacksquare}$

\begin{figure}[h]
\null\hfill
\subfigure[]{
	\begin{pspicture}(0,0)(3,2)
	\rput[bl](0,0){\scalebox{0.5}{\includegraphics{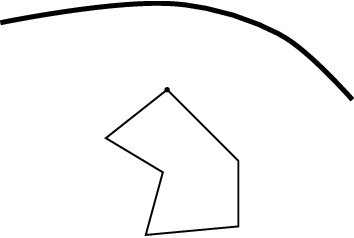}}}
	\rput[bl](2.2,0.4){$\gamma_{i_0}$} \rput[bl](0.4,1.6){$\partial S$} \rput(1.4,1.4){$v$}
	\end{pspicture}}
\hfill
\subfigure[]{
	\begin{pspicture}(0,0)(3.5,2)
	\rput[bl](0,0){\scalebox{0.5}{\includegraphics{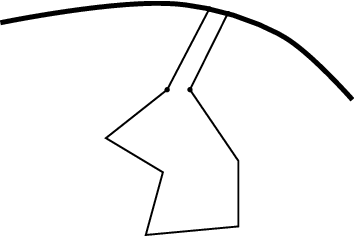}}}
	\rput[bl](2.2,0.4){$\gamma \sim \gamma_{i_0}$} \rput[bl](0.4,1.6){$\partial S$} \rput(1.2,1.4){$v'$} \rput[l](1.8,1.4){$v''$}
	\end{pspicture}}
\hfill\null
\caption{\label{fig:closed_not}}
\end{figure}
\smallskip

Now let $E$ be a thin strip along $\gamma$ and, as usual, denote by $S'$ the manifold $S$ with $E$ removed (this is an operation of type $R_1$). The sequence of Proposition \ref{prop:exacta} reads \[H_1(E,E \cap \partial S) = \mathbb{Z}_2 \stackrel{j_*}{\longrightarrow} H_1(S-L,\partial S) \stackrel{f}{\longrightarrow} H_1(S'-L,\partial S') \longrightarrow 0,\] where $H_1(E,E \cap \partial S)$ is actually generated by $[\gamma]$. Again $[\gamma] \in {\rm im}\ j_*$ and therefore $r'_1 = r_1 - 1$ whereas $r'_i = r_i$ for $i = 0,2$.

Repeating this $r_1$ times leads to a surface $S^{(r_1)}$ which is obtained from $S$ by a sequence of $r_1$ operations of type $R_1$ and has $\check{H}^1(S^{(r_1)},L) = 0$. As before, to simplify the notation we denote $S^{(r_1)}$ by $S$. This concludes the second step.
\medskip

{\it Step 3.} To complete the construction of $P$ it only remains to observe that $r_2 = \check{H}^0(S,L)$ is the number of connected components of $S$ that do not meet $L$. Then by removing those components with an $R_2$ operation for each we obtain a new $S^{(r_2)}$ with $\check{H}^0(S^{(r_2)},L) = 0$. Clearly we still have $\check{H}^i(S^{(r_2)},L) = 0$ for $i = 1,2$. Setting $P := S^{(r_2)}$ concludes the proof of the theorem. Indeed, by construction we have $\check{H}^*(P-L,\partial P;\mathbb{Z}_2) = 0$ which shows that each component of $P-L$ contains exactly one component of $\partial P$. Then arguing as in the Claim of page \pageref{pg:punctured} each component of $P-L$ is a once-punctured closed disk. The result follows.

\bibliographystyle{plain}
\bibliography{regular}

\begin{thebibliography}{10}

\bibitem{edwards2}
Jr. C.~H.~Edwards.
\newblock Concentricity in $3$--manifolds.
\newblock {\em Trans. Amer. Math. Soc.}, 113(3):406--423, 1963.

\bibitem{chewningowen1}
W.~C. Chewning and R.~S. Owen.
\newblock Local sections of flows on manifolds.
\newblock {\em Proc. Amer. Math. Soc.}, 49:71--77, 1975.

\bibitem{churchill1}
R.~C. Churchill.
\newblock Isolated invariant sets in compact metric spaces.
\newblock {\em J. Diff. Eq.}, 12:330--352, 1972.

\bibitem{conley1}
C.~Conley.
\newblock {\em Isolated invariant sets and the {M}orse index}, volume~38 of
  {\em CBMS Regional Conference Series in Mathematics}.
\newblock American Mathematical Society, 1978.

\bibitem{conleyeaston1}
C.~Conley and R.~Easton.
\newblock Isolated invariant sets and isolating blocks.
\newblock {\em Trans. Amer. Math. Soc.}, 158:35--61, 1971.

\bibitem{easton2}
R.~Easton.
\newblock {Flows near isolated invariant sets in dimension $3$}.
\newblock In G.~E.~O. Giacaglia, editor, {\em Periodic orbits, stability and
  resonances}, pages 332--336. {D. Reidel Publishing Company}, 1970.

\bibitem{gierzkiewicz1}
A.~Gierzkiewicz and K.~W{\'{o}}jcik.
\newblock On the cohomology of an isolating block and its invariant part.
\newblock {\em Topol. Methods Nonlinear Anal.}, 32(2):313--326, 2008.

\bibitem{mischaikow}
K.~Mischaikow.
\newblock Topological techniques for efficient rigorous computation in
  dynamics.
\newblock {\em Acta Numerica}, 11:435--477, 2002.

\bibitem{moise2}
E.~E. Moise.
\newblock {\em Geometric topology in dimensions 2 and 3}.
\newblock Springer-Verlag, 1977.

\bibitem{richards1}
I.~Richards.
\newblock On the classification of noncompact surfaces.
\newblock {\em Trans. Amer. Math. Soc.}, 106(2):259--269, 1963.

\bibitem{rolfsen1}
D.~Rolfsen.
\newblock {\em Knots and links}.
\newblock AMS Chelsea Publishing, 2003.

\bibitem{rourkesanderson1}
C.~P. Rourke and B.~J. Sanderson.
\newblock {\em Introduction to piecewise-linear topology}, volume~69 of {\em
  Ergebnisse der Mathematik und ihrer Grenzgebiete}.
\newblock Springer-Verlag, 1972.

\bibitem{salamon1}
D.~Salamon.
\newblock Connected simple systems and the {C}onley index of isolated invariant
  sets.
\newblock {\em Trans. Amer. Math. Soc.}, 291(1):1--41, 1985.

\bibitem{mio2}
J.~J. S{\'a}nchez-Gabites.
\newblock {How strange can an attractor for a dynamical system in a
  $3$--manifold look?}
\newblock {\em Nonlinear Anal.}, 74(17):6162--6185, 2011.

\bibitem{mio1}
J.~J. S{\'a}nchez-Gabites.
\newblock {Arcs, balls and spheres that cannot be attractors in
  $\mathbb{R}^3$}.
\newblock {\em Trans. Amer. Math. Soc.}, 368(5):3591--3627, 2016.

\bibitem{spanier1}
E.~H. Spanier.
\newblock {\em Algebraic topology}.
\newblock McGraw--Hill Book Co., 1966.

\bibitem{wazewski1}
T.~Wa\.zewski.
\newblock Sur un principe topologique de l'examen de l'allure asumptotique des
  int\'egrales des \'equations differenti\'elles ordinaires.
\newblock {\em Ann. Soc. Polon. Math.}, 20:279--313, 1947.

\end{thebibliography}

\end{document}